\documentclass{amsproc}
\usepackage{epsfig, amsmath, amssymb,  amsthm, 
	times}
\usepackage{pdfpages}

\usepackage[utf8]{inputenc}

\usepackage{amsmath,amssymb,amsfonts,
	bbm,mathtools}
\usepackage{mathabx}
\usepackage{graphicx}
\usepackage{subfigure}
\usepackage{cite}
\usepackage{hyperref}
\usepackage{url}
\usepackage{mathrsfs}  
\usepackage{tikz}
 
 \usepackage{mathtools}
\usepackage{soul}
\usepackage{float}
\usetikzlibrary{spy}

\usepackage[T1]{fontenc}
\usepackage{esint}

\numberwithin{equation}{section}

\newtheorem{thm}{Theorem}[section]

\newtheorem{pr}[thm]{Proposition}
\newtheorem{cor}[thm]{Corollary}
\newtheorem{conj}[thm]{Conjecture}
\newtheorem{df}[thm]{Definition}
\newtheorem{rem}[thm]{Remark}

\newtheorem{main-idea}[thm]{Main Ideas}
\newtheorem{main-result}[thm]{Main results}

\def\R{{\mathbb R}}

\def\be{\begin{equation}}
	\def\ee{\end{equation}}

\def\1{\operatorname{\mathbf 1}}

\DeclareMathOperator{\AC}{AC}
\DeclareMathOperator{\dom}{dom}
\DeclareMathOperator{\Var}{Var}
\DeclareMathOperator{\diam}{diam}

\def\mint{\mathop{\,\rlap{-}\!\!\int}\nolimits}

\numberwithin{equation}{section}

\newcommand{\abs}[1]{\left\vert#1\right\vert}
\newcommand{\norm}[1]{\left\Vert#1\right\Vert}

\newcommand{\www}{0.49\textwidth}

\begin{document}
\title{Convergence, optimization and stability of singular eigenmaps}

\author[{Akwei \emph{et al.}}]
{
Bernard Akwei, 
Bobita Atkins, 
Rachel Bailey, 
Ashka Dalal, 
Natalie Dinin, 
Jonathan Kerby-White, 
Tess McGuinness,  
Tonya Patricks, 
Luke Rogers, 
Genevieve Romanelli, 
Yiheng Su, 
Alexander Teplyaev 
	}

	\address{Department of Mathematics, 
		University of Connecticut, 
		Storrs CT 06269 USA 
		}

\email{bernard.akwei@uconn.edu}
\email{rachel.bailey@uconn.edu}
\email{rbailey@bentley.edu}
\email{luke.rogers@uconn.edu}
\email{alexander.teplyaev@uconn.edu}


\begin{abstract}
	Eigenmaps are important in analysis, geometry, and machine learning, especially in nonlinear dimension reduction. Approximation of the eigenmaps of a Laplace operator depends crucially on the scaling parameter $\epsilon$. If $\epsilon$ is too small or too large, then the approximation is inaccurate or completely breaks down. However, an analytic expression for the optimal $\epsilon$ is out of reach. In our work, we use some explicitly solvable models and Monte Carlo simulations to find the approximately optimal range of $\epsilon$ that gives, on average, relatively accurate approximation of the eigenmaps. Numerically we can consider several model situations where eigenmaps can be computed analytically, including intervals with uniform and weighted measures, squares, tori, spheres, and the Sierpinski gasket. In broader terms, we intend to study eigenmaps on weighted Riemannian manifolds, possibly with boundary, and on some metric measure spaces, such as fractals.	
	\tableofcontents
\end{abstract}

		\maketitle

 \section{Introduction}There is an extensive mathematical literature on eigenmaps in analysis, geometry, and applied mathematics. Versions of the Laplacian eigenmaps introduced by Belkin and Niyogi are widely used in nonlinear dimension reduction techniques in data analysis and machine learning. The reader can find relevant information and further references in \cite{belkin2003laplacian, belkin2008towards, coifman2006diffusion, koltchinskii2000random, gine2006empirical, burago2015graph, lin2022varadhan, green2021minimax, jones2008manifold, garcia2020error, Venkatraman2023} and particularly relevant  results   in the recent article of Martin Wahl \cite{wahl2024kernel}. 
In our work we study eigenmaps on an wide variety of spaces, with particular attention to examples that are not covered by the existing theory, including   simple Riemannian manifolds with boundary and some fractal spaces.

The use of eigenmaps for dimension reduction problems was first proposed in~\cite{belkin2003laplacian}.  The suggestion was that a set of data points $X=\{x_1,\dotsc,x_j\}$ in a high dimensional space $\mathbb{R}^N$ should be  treated as vertices of a graph, either by taking edges between points separated by distance at most a threshold $\epsilon$ or by joining each vertex to its $k$ nearest neighbors.   A small number $D$ of eigenfunctions $\phi_1,\dotsc,\phi_D$ of the graph Laplacian (or a weighted graph Laplacian) could then taken as coordinates for the data, defining an eigenmap $\Phi=(\phi_1,\dotsc,\phi_D):X\to\mathbb{R}^D$. This method was motivated by arguments suggesting that if the original data consisted of $n$ sufficiently well-distributed points on a nice manifold $M$, then the eigenmap would preserve geometric features of $M$.

An associated collection of mathematical problems may be stated, somewhat vaguely, as follows:  Suppose $X=\{x_1,x_2,x_3,\dotsc\}$ are points on a Riemannian manifold $M$ with Laplace-Beltrami operator $\mathcal{L}$, and one defines a graph Laplacian $\mathcal{L}_n$ by defining edge weights between pairs of points in $X_n=\{x_1,\dotsc,x_n\}$.  Under what conditions on the distribution of the points in $X$, the nature of the edge weights, and the geometric properties of $M$ can one ensure that $\mathcal{L}_n$ converges to $\mathcal{L}$?  Note that in this formulation the geometric features of $M$ are encoded in its Laplace-Beltrami operator  and the sense and extent to which these are preserved are encoded in the type and rate of convergence of $\mathcal{L}_n$ to $\mathcal{L}$.

Several authors have developed rigorous results on the geometric properties of eigenmaps by proving theorems that solve problems of the above type.  These have involved various assumptions on the distribution of points, as well as hypotheses about the smoothness of the manifold and bounds on its curvature. One possible approach is presented in Appendix~\ref{ARTsection} below. A common feature of these results that the approximation proceeds in two steps: one first shows that under smoothness and curvature assumptions one can approximate the Laplace-Beltrami operator of $M$ by an operator giving the $\epsilon$-normalized difference of the function value and a weighted average over balls of a sufficiently small size $\epsilon$, and then shows that this difference operator can be approximated by graph Laplacian operators provided the sample points are sufficiently well-distributed. The $\epsilon$-normalization contains a parameter $\beta$ reflecting how smoothness is connected to length scaling; on manifolds we have $\beta=2$, but this is not the case for general metric spaces.  In what follows, it will be helpful to distinguish these steps, approximating $\mathcal{L}$ by an average difference operator $\mathcal{L}_\epsilon$ and then $\mathcal{L_\epsilon}$ by a sequence of graph Laplacian operators $\mathcal{L}_{\epsilon,n}$.  The two steps of the approximation each correspond to a term in the triangle inequality
\begin{equation}\label{eqn:triangleest}
|\mathcal{L}f-\epsilon^{-\beta}\mathcal{L}_{\epsilon,n}f|\leq |\mathcal{L}f-\epsilon^{-\beta}\mathcal{L}_\epsilon f|+\epsilon^{-\beta}|\mathcal{L}_\epsilon f-\mathcal{L}_{\epsilon,n}f|
\end{equation}
and it is significant that a good estimate depends  crucially on the choice of scale parameter $\epsilon=\epsilon(n)$. If $\epsilon$ is too small or too large then the approximation is inaccurate or breaks down completely.  One is typically  interested in an $\epsilon$ that is suitable for use across a space of functions (e.g., with some degree of smoothness) but an analytical expression for an optimal $\epsilon$ across such a class appears to be difficult to establish in general settings.

In the present work our goal is to present several model situations where eigenmaps can be computed analytically or numerically, and which illuminate cases where the existing theory is not applicable or is incomplete.  Our elementary examples include intervals with uniform (Figure~\ref{uni-fig} and \ref{fig-best-fit}) and weighted measures (Figure~\ref{gauss-exp-fig}), squares, tori (Figure~\ref{sttt-fig}), spheres, and the Sierpinski gasket (Figure~\ref{fractal-fig}).  Intervals and squares are among the simplest examples in which  the underlying space is a manifold with boundary, while the existing literature assumes either that there is no boundary or that one is well-separated from it.  We also investigate a connection between eigenmaps and orthogonal polynomials in the former setting.  The Sierpinski Gasket is one of the simplest fractal spaces that admits a natural Laplacian which corresponds to a very different notion of smoothness than that found on a manifold; in particular, the associated average difference operator has a different scaling than is considered in the existing theory. We also use both our explicitly solvable models and Monte Carlo simulations to  investigate the optimal range $\epsilon(n)$ to use in the estimate~\ref{eqn:triangleest} for a given $n$ and prescribed point distribution, and to consider the dependence and stability of the method when the choice of Laplacian is varied. These examples are intended to serve as model cases for later research on the corresponding problems for eigenmaps on weighted Riemannian manifolds, possibly with boundary, and on some metric measure spaces, especially self-similar fractals known to support Dirichlet forms, such as the nested fractals introduced by Lindstr\"om~\cite{Lindstrom}.
In addition to theoretical results and numeriacl illustrations, we formulate eight conjectures 
\ref{conj:betterconvergence},
\ref{conj-BN-GK},
\ref{conj-BGLK},
\ref{Sierpinski-conj},
\ref{conj:SGlapbyrandgraphs},
\ref{conj-conv-rot},
\ref{conj-conv2} and 
\ref{conj-no-conv3}
dealing with specifics of convergence of eigenmaps on various spaces. 

\subsection*{Acknowledgments} This research is supported in part by NSF grants DMS~1950543 and DMS~2349433. The last author acknowledges with appreciation constructive and very helpful discussions with Michael Hinz that significantly enhanced this work.

 \begin{figure}[h]
 	\centering
 	\includegraphics[width=\www]{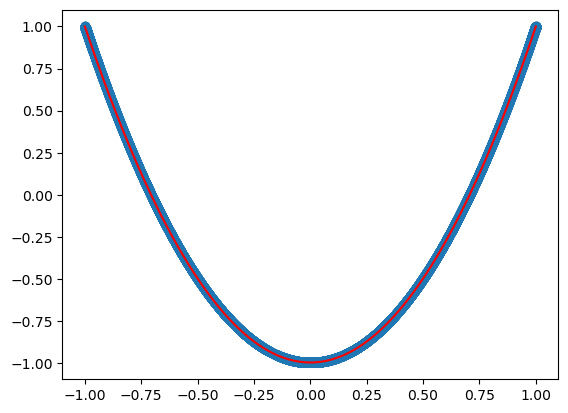}
 	\includegraphics[width=\www]{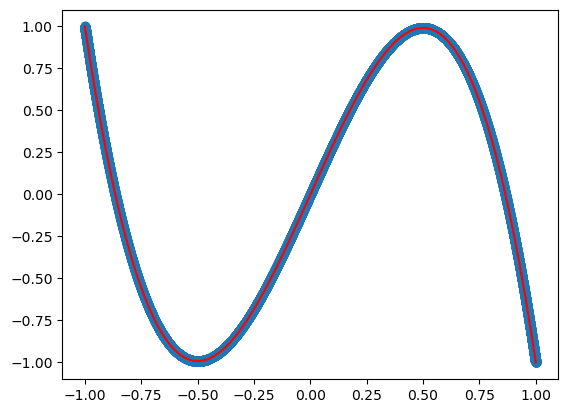}

 	\includegraphics[width=\www]{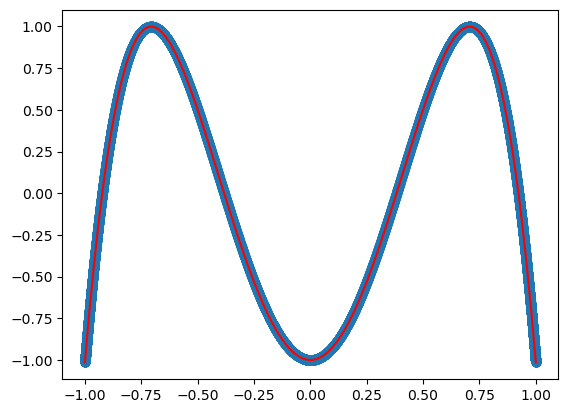}
 	\includegraphics[width=\www]{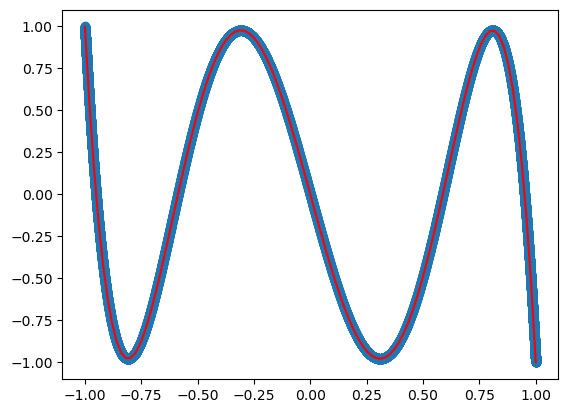}
 	\caption{Polynomial best fit curves to random eigenmaps:
 		\newline
 		$y = 2.0036x^2 + 8\cdot10^{-11}x -1.0000$ $\approx$ $T_2(x)=2x^2 -1$; 
 		\newline
 		$y = 3.9778x^3 + 10^{-10}x^2 -2.9780x -3\cdot10^{-11}$ $\approx$ $T_3(x) =4x^3 - 3x$; 
 		\newline
 		$y = 
 		+8.0288x^4 
 		+3\cdot10^{-9}x^3 
 		- 8.0149x^2 
 		- 10^{-10}x 
 		+1.0002$ \newline
 		 $\approx$ $T_4(x)= 8x^4-8x^2+1$;
 		\newline
 		$y = 15.7055x^5 -
 		6\cdot10^{-8}x^4 -19.5986x^3 + 4\cdot10^{-8}x^2 + 4.8916x -2\cdot10^{-9}$  $\approx$ 
 		$T_5(x)=16x^5-20x^3+5x$.}
 	\label{fig-best-fit}\end{figure}

  \begin{figure}[h] 
  	\centering
  	
  	\includegraphics[width=\www]{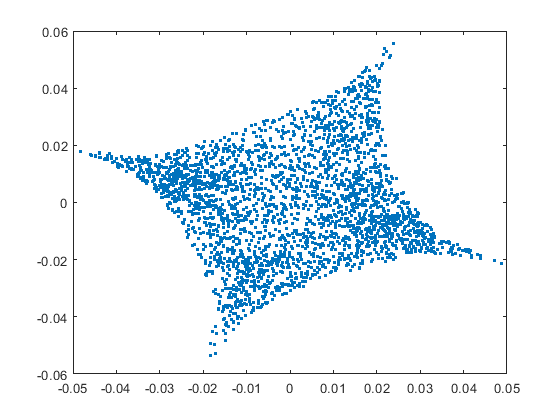}
  	\includegraphics[width=\www]{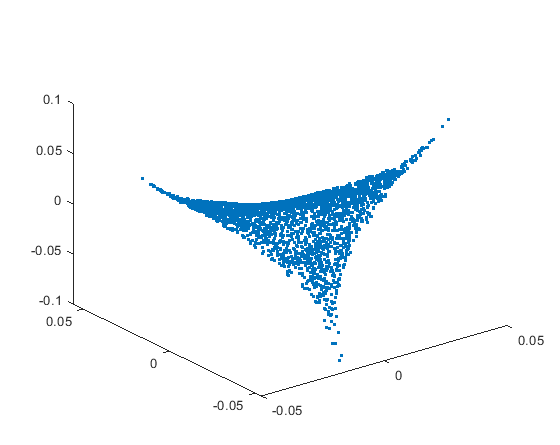}
  	
  	\includegraphics[width=\www]{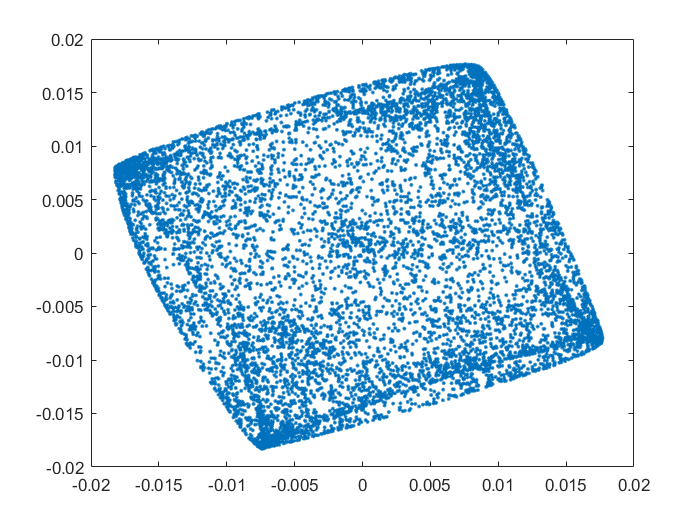}
  	\includegraphics[width=\www]{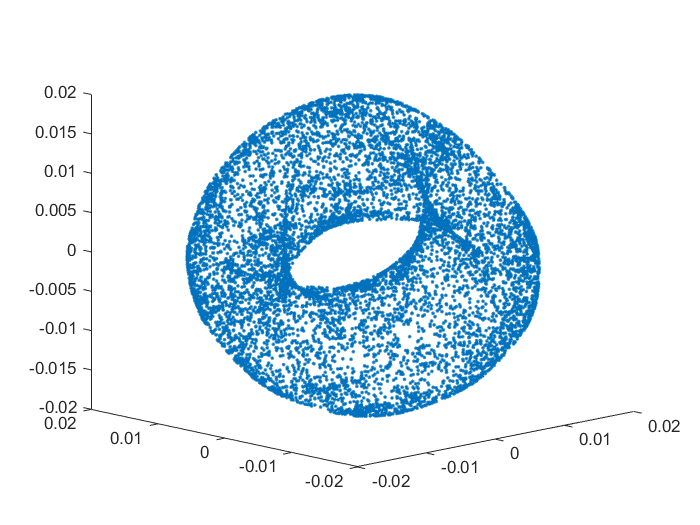}
  	
  	\includegraphics[width=\www]{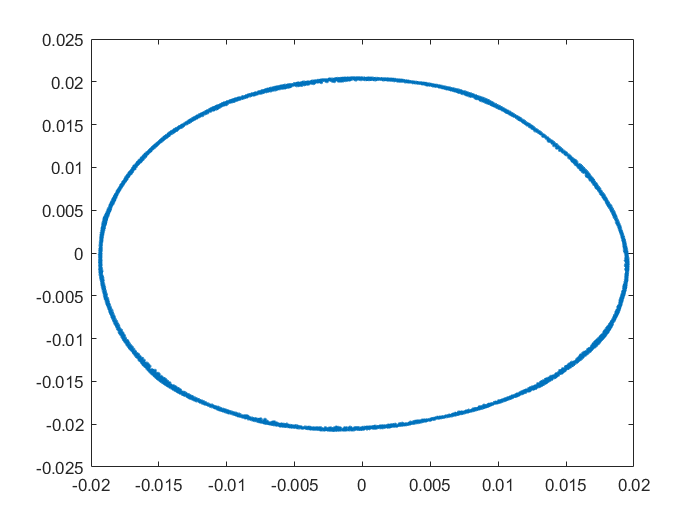}
  	\includegraphics[width=\www]{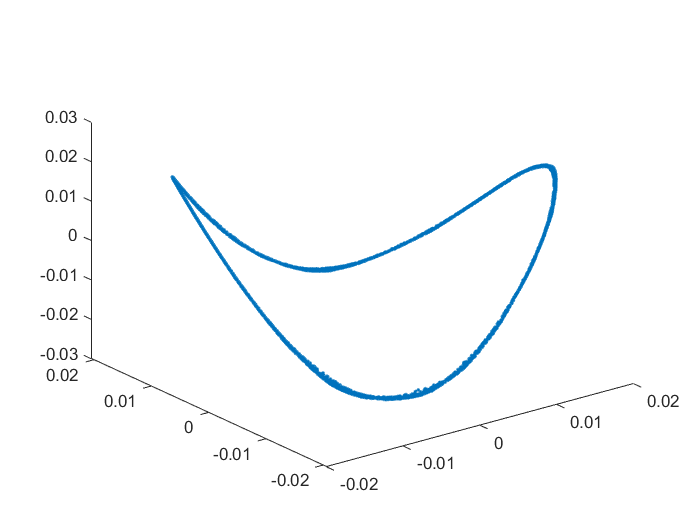}
  	
  	\caption{
  		Illustrations to  2- and  3-dimensional eigenmaps:       a square with  $n = 3000, \varepsilon = 0.7$ (top);  a torus with $n = 10000, \varepsilon = 0.5$ (middle);   a thin   torus with $n = 5000, \varepsilon = 0.1$ (bottom).}
  	\label{sttt-fig}\end{figure}

\section{Framework}

\subsection{Laplacian Operators}

Our setting is a metric space $(X,d)$ with a Borel   measure $\mu$.  We are interested in several types of ``Laplacian'' operators that can be applied to functions on $X$.  The first is a generalization of the Laplace-Beltrami operator on a manifold.  Although these operators have a deep and complicated theory, we will mostly concern ourselves with the case of   weighted Laplace-Beltrami operators, possibly with boundaries, on manifolds and with Laplacians associated to resistance forms in the sense of~\cite{Kigamibook,FOT,BGL}.

\begin{df}\label{def:lap}
A Laplacian $\mathcal{L}$ on $(X,d,\mu)$ is a non-positive self-adjoint Markovian operator with dense domain in $L^2(\mu)$; $\mathcal{L}$  generates a (heat) semigroup $P_t=e^{t\mathcal{L}}$. 
\end{df}

The second type of operator we consider gives the difference between the function value at a point and its average over a ball. The average may involve a kernel.  We call the result an averaging Laplacian.
\begin{df}\label{def:avlap}
Our averaging  Laplacians have the following forms
\begin{itemize}
\item For $\epsilon>0$ and each $x\in X$ specify a set $N_x$ containing $x$ and of diameter at most $2\epsilon$, and set
\begin{equation*}
\mathcal{L}_\epsilon = \frac{1}{\mu(N_x)} \int_{N_x} \bigl( u(y)-u(x) \bigr) d\mu(y).
\end{equation*}
We frequently use $N_x=B(x,\epsilon)$ the ball of radius $\epsilon$ centered at $x$, but sometimes other choices are more natural.
\item For a kernel  $K(x,y)\geq0$ so $\int K(x,y)d\mu(y)=1$ let 
\begin{equation*}
\mathcal{L}_K = \int K(x,y) \bigl( u(y)-u(x) \bigr) d\mu(y).
\end{equation*}
\item In the special case that $K(x,y)=k(d(x,y))$ set
\begin{equation*}
 K_\epsilon(x,y)=\frac{k(\epsilon^{-1}d(x,y))}{\int k(\epsilon^{-1}d(x,y)) d\mu(y)}
\end{equation*}
and write $\mathcal{L}_{K,\epsilon}$ instead of $\mathcal{L}_{K_\epsilon}$.
\end{itemize}
Note that if the $N_x$ are balls of radius $\epsilon$ then $\mathcal{L_\epsilon}=\mathcal{L_{K,\epsilon}}$ for $K(x,y)=\mu(B(x,1))^{-1} \mathbf{1}_{B(x,1)}(y)$.
\end{df} 

The remaining types of operators we consider are graph Laplacians.
\begin{df}\label{def:abstractgraphlap}
A graph is a finite set of vertices $\{x_1,\dotsc,x_n\}$ with edge weights $w_{i,j}\geq0$ from $x_i$ to $x_j$ for $i\neq j$.  We let $w_{i,i}=-\sum_{j=1}^n w_{i,j}$ and define the graph Laplacian to be the matrix $L=-[w_{i,j}]_{i,j=1}^n$.  Evidently $L$ determines and is determined by the graph, and by definition $L\mathbf{1}=0$ so constants are in the kernel of $L$.
\end{df}

\begin{df}\label{def:graphlaplacians}
For a finite set $S_n=\{x_1,\dotsc,x_n\}\subset X$ we consider the following graph Laplacians.  Notice that they are defined at all $x\in X$ rather than just at the set $S_n$, because at each $x$ we consider a graph on $S_n\cup\{x\}$.  However, this does mean that the graph used at $x$ depends on $x$.
\begin{itemize}
\item For $\epsilon>0$ let $E_\epsilon(x)=\{x_j\in S_n:d(x,x_j)\leq\epsilon\}$. Provided $\epsilon$ is large enough that $E_\epsilon(x)$ is non-empty for each $x$, define
\begin{equation*}
\mathcal{L}_{\epsilon,n} f(x) =  \frac1{\# E_\epsilon(x)} \sum_{E_\epsilon(x)} f(x_j)-f(x).
\end{equation*}
Note that this is simply the graph Laplacian on $S_n\cup\{x_0\}$ with $x_0=x$ and with edges from between $x_i$ to $x_j$ given by $w_{i,j}=(\# E_\epsilon(x))^{-1}$ if $d(x,y)\leq \epsilon$ and $w_{i,j}=0$ otherwise.
\item If $K(x,y)\geq 0$ is a kernel  we define $L_{K,n}$ to be the graph Laplacian on $S_n\cup\{x_0\}$ with weights $w_{i,j}=\Bigl( \sum_{l=1}^n K(x_i,x_l) \Bigr)^{-1}K(x_i,x_j)$ for $i\neq j$.
\item If $K$ is a kernel so $K(x,y)=k(d(x,y))$ we define $\mathcal{L}_{K,\epsilon,n} = \mathcal{L}_{K_\epsilon}$  where 
\begin{equation*}
K_\epsilon(x_i,x_j)=\Bigl( \sum_{l=1}^n k(\epsilon^{-1}d(x_i,x_l)) \Bigr)^{-1} k(\epsilon^{-1}d(x_i,x_j)).
\end{equation*}
\end{itemize}
In the original literature the kernel $K$ is often taken to be a Gaussian.  We often think of the situation where $S_n$ as the first $n$ elements of a sample set $S$, so that the above defines a sequence of operators indexed by $n$.
\end{df}

Finally, we consider the situation where the sample sets $S_n$ are random.
\begin{df}\label{def:randomgraphLap}
If $(x_1,\dotsc,x_n)$ is a random variable on $(\Omega,d\mathbb{P})$ with values in $X^n$ we consider $S_n(\omega)=(x_1,\dotsc,x_n)(\omega)$ as an $n$-point random sample from $X$ and define operators $\mathcal{L}_{\epsilon,n}(\omega)$, $\mathcal{L}_{K,n}(\omega)$ and $\mathcal{L}_{K,\epsilon,n}(\omega)$ as in the preceding definition.
\end{df}

\subsection{Main Questions and some known results}\label{ssec:questionsandresults}

As described in the introduction, the main theoretical question of interest is to describe circumstances under which we can approximate a Laplacian operator $\mathcal{L}$ by graph Laplacians.  This is usually done by approximating $\mathcal{L}$ by (rescaled) averaging Laplacians $\mathcal{L}_{K,\epsilon}$ and then  approximating  $\mathcal{L}_{K,\epsilon}$  by a sequence of graph Laplacians $\mathcal{L}_{K,\epsilon,n}$ or random graph Laplacians $\mathcal{L}_{K,\epsilon,n}(\omega)$.  Several main questions arise.

\begin{enumerate}
\item Under what circumstances can one rescale averaging Laplacians $\mathcal{L}_{K,\epsilon}$ so that they converge to a Laplacian $\mathcal{L}$?  This is known to be a very difficult problem in general.  In the case that $\mathcal{L}$ is a Laplace-Beltrami operator on a Riemannian manifold with some additional (quite non-trivial) geometric structure, one finds that for suitable kernels $K$ there is convergence of $\epsilon^{-2}\mathcal{L}_{K,\epsilon}f$ to $\mathcal{L}f$ for functions $f$ that are sufficiently smooth.  There are other examples of metric measure spaces for which the ``correct'' rescaling is $\epsilon^{-\beta}\mathcal{L}_{K,\epsilon}f\to\mathcal{L}f$ where $\beta\neq 2$; again, this is known for specific kernels, and it is significant that the simplest kernel involving the characteristic function of a ball with respect to the original metric may not be of the correct kind   In both cases the notion of ``smoothness'' of $f$ may be made using the operator $\mathcal{L}$, see Section~\ref{fractalsection} and Appendix~\ref{ARTsection}.
\item Under what circumstances will a sequence $\mathcal{L}_{K,\epsilon,n}f$ converge to  $\mathcal{L}_{K,\epsilon}f$?  This again involves the smoothness of the function $f$ in some manner, because we must approximate the averaged difference in $\mathcal{L}_{K,\epsilon}f$ by the discrete sampled average  $\mathcal{L}_{K,\epsilon,n}f$.  Both the nature of the kernel $K$ and the sampling set can play significant roles; typically their properties are related to one another. An example of a simple result of this kind is in Section~\ref{sec:graphtoaveraging} below.
\item If one can perform both of the preceding approximations for suitably smooth $f$, so that in \eqref{eqn:triangleest}
can one choose $\epsilon=\epsilon(n)$ so that $\epsilon(n)^{-\beta} \mathcal{L}_{K,\epsilon(n),n}f\to\mathcal{L}$ as $n\to\infty$?  If so, (how) can one find an optimal choice of $\epsilon(n)$?  Are there circumstances under which one can provide a sequence $\epsilon(n)$ so that the preceding convergence occurs with a guaranteed rate?
\item In the original problem of dimension reduction of a data set by a Laplacian eigenmap, the motivating idea was that mapping via a small number of eigenfunctions preserves geometric features of the set.  From this perspective, the important notion of convergence of $\epsilon^{-\beta} \mathcal{L}_{K,\epsilon,n}$ to $\mathcal{L}$ is that $\epsilon^{-\beta} \mathcal{L}_{K,\epsilon,n}f\to\mathcal{L}f$ for $f$ in the span of the first $m$  eigenfunctions, where $m$ is prescribed in advance.  Can any of the preceding questions about optimal $\epsilon(n)$ or rates of convergence be answered for this possibly simpler class of functions?
\item Considering a random sample set $S(w)$, what conditions on the sampling distribution provide convergence of the preceding types, and how do the choices of $\epsilon(n)$ or the bounds on rates of convergence depend on the this distribution?
\item Are there circumstances in which one has convergence of  $\epsilon^{-\beta} \mathcal{L}_{K,\epsilon(n),n}$ to $\mathcal{L}$ without an intermediate  $\epsilon^{-\beta} \mathcal{L}_{K,\epsilon}$ operator being useful/interesting/natural etc?
\end{enumerate}

There are many known results regarding these questions.  One of the simplest involves fixing a point $x$ and considering the convergence of a random graph Laplacian to the averaging Laplacian at $x$, for which one has a central limit theorem.  Note that this result is at a fixed $x$ and it is considerably more difficult to give a similar result that is uniform over $x$.
\begin{thm}\label{thm-monte-carlo}
Let $(\Omega, \mathcal{F}, \mu)$ be a probability space and let $K:\Omega \times \Omega \to \mathbb{R}$ be bounded.  Suppose $x_1,\dots, x_n$ are i.i.d random variables and define  for $f \in L^2(\Omega)$
\begin{equation*}
L_{K,n}f(x):=\frac{1}{n}\sum_j^n K(x,X_j)(f(x)-f(y)).
\end{equation*}
Then one has the following convergence in distribution to a standard normal $Z\sim N(0,1)$.
\begin{equation*}
	\frac{\sqrt{n}(L_{K,n}f(x)-\mathbb{E}(K(x,X_1)(f(X_1)-f(x))))}{\sqrt{\Var(K(x,X_1)(f(X_1)-f(x)))}} \to Z.
\end{equation*}
\end{thm}
\begin{proof}
Let $y_j = K(x,y_j)(f(y_j)-f(x))$, so the $y_j$ are i.i.d random variables and both $\mathbb{E}(y_1)=\mathbb{E}(K(x,x_1)(f(x_1)-f(x))$ and $\Var(Y_1)<\infty$.  The conclusion follows from Lindeberg-L\'evy CLT.
\end{proof}

Many of the earliest results regarding the performance of Laplacian eigenmaps involved random graph Laplacians with Gaussian edge weights on a $d$-dimensional Riemannian manifold with Laplace-Beltrami operator $\mathcal{L}$.  In our notation from Definition~\ref{def:graphlaplacians}, one has i.i.d. sample points $x_1,\dotsc, x_n$, a Gaussian kernel $K$ and the random graph Laplacian $\mathcal{L}_{K,\epsilon,n}(\omega)$, where usually $\epsilon=\epsilon(n)$, and the quantity to be estimated is
\begin{equation}\label{eqn:quantitytoestimate}
\bigl| \epsilon(n)^{-2}\mathcal{L}_{K,\epsilon(n),n}(\omega)f(x) - c\mathcal{L}f(x)\bigr|
\end{equation}
where $c$ is a normalizing constant.

In a paper~\cite{belkin2008towards} based on the thesis of Belkin,  
Belkin and Niyogi show that with $\epsilon(n)\asymp n^{-\frac2{d+2+\alpha}}$ one has convergence of~\eqref{eqn:quantitytoestimate} to zero in probabilty. Moreover, they show this estimate can be made uniform over $f$ in a  family of smooth functions with bounded derivatives of degree up to three at the cost of restricting to a subsequence.

At around the same time, Gine and Koltchinskii~\cite{gine2006empirical} studied a similar problem. Under the hypotheses
\begin{align*}
\frac{n \epsilon(n)^{d+2}}{-\log \epsilon(n)}\to\infty \quad \text{ and }\quad \frac{n\epsilon(n)^{d+4}}{-\log \epsilon(n)}\to 0
\end{align*}
they showed that, uniformly over both $f$ smooth with bounded derivatives of order up to three and uniformly over $x\in X$, one has almost surely the following bound for the quantity in~\eqref{eqn:quantitytoestimate}:
\begin{equation}\label{eqn:GKestimate}
	\bigl| \mathcal{L}_{K,\epsilon(n),n}(\omega)f(x) - c\mathcal{L}f(x)\bigr|= O\Bigl( \sqrt{\frac{-\log \epsilon(n)} {  n\epsilon(n)^{d+2}} }\Bigr).
\end{equation}
They also provide a functional central limit theorem for $\mathcal{L}_{K,\epsilon(n),n}(\omega)f(x) - c\mathcal{L}f(x)$ at each fixed $x$.

Earlier work of the same authors~\cite{koltchinskii2000random} is also significant in this context. They deal with the question of approximating a Hilbert-Schmidt operator $Hf=\int h(x,y)f(y)d\mu(y)$ on a probability space $(X,\mu)$, by taking i.i.d.\ samples $\{x_1,\dotsc,x_n\}$ and forming $H_n=\frac1n(1-\delta_{j,k})H(x_j,x_k)$ (the empirical matrix associated to $h(x,y)$ with the diagonal deleted).  The main result is that the spectrum of $H$ can be approximated by that of $H_n$.  Under some assumptions on the spectrum, the results include a law of large numbers, estimates on convergence rates, and a central limit theorem.  There is related work on spectral projectors~\cite{Koltchinskii} in the same context, see also~\cite{KoltchinskiiLounici}.

It is apparent that estimates on~\eqref{eqn:quantitytoestimate} are, in a sense, too strong for the purpose of understanding the performance of Laplacian eigenmaps, where one is primarily interested in the approximation of eigenfunctions with low eigenvalues, or equivalently spectral projectors onto the corresponding eigenspaces, rather than pointwise approximation estimates for the Laplacian.  If one permits an averaging Laplacian $\mathcal{L}_{K,\epsilon}$ in which the kernel is adapted to the space of interest, and considers its approximation by a suitably defined graph Laplacian $\mathcal{L}_{K,\epsilon,n}$ then one has much better estimates for $\mathcal{L}_{K,\epsilon}-\mathcal{L}_{K,\epsilon,n}$.  One of the main results of~\cite{rosasco-belkin-2010learning,koltchinskii2000random} is an estimate of this difference in Hilbert-Schmidt norm when $K$ is the reproducing kernel of a reproducing kernel Hilbert space that can be realized as a Sobolev space on the underlying manifold (see Theorem~15 of~\cite{rosasco-belkin-2010learning,koltchinskii2000random}).  The proof is based on concentration inequalities for Hilbert spaces.

An alternative approach to understanding the convergence of spectral projectors for graph Laplacians to that of the Laplace-Beltrami operator on a Riemannian manifold  of dimension $d$ has been taken by Burago, Ivanov and Kurylev~\cite{burago2015graph}.  They show that both the first $k$ eigenvalues and eigenfunctions of the Laplace-Beltrami operator can be approximated within $C(k,d)(\epsilon+ \frac r\epsilon)$, provided $\epsilon+ \frac r\epsilon$ is sufficiently small and $r\leq c(k,d)\epsilon$, by the eigenvalues and eigenfunctions of a weighted graph Laplacian $\mathcal{L}_{\epsilon,n}$.  They require that the graph vertices $x_1,\dotsc,x_n$ are well spaced in the sense that the manifold can be partitioned into equal measure subsets each lying in a single $B(x_j,r)$.  Since the number of partition sets is a multiple of $r^{-d}$, this requires $n\asymp r^{-d}$, so the error is bounded by $\epsilon(n)+ \frac{1}{\epsilon(n)n^{1/d}}$ and provides a bound for the error of size $n^{-\frac1{2d}}$.

In a subsequent work~\cite{burago2019spectral}, Burago, Ivanov and Kurylev consider an operator akin to $\epsilon^{-2}\mathcal{L}_\epsilon$ and give estimates of how its spectrum changes under certain perturbations of the underlying metric measure space. In particular, this permits approximation of the spectrum of this operator (which we would call a rescaled averaging Laplacian) by the corresponding operator for a discrete metric space, namely a graph Laplacian, using estimates on how close the discrete metric space is to the original one in the sense of Fukaya.  The techniques involve measure transportation on the underlying metric space, and should presumably generalize to averaging Laplacians with other rescalings, as may be expected to occur on more general metric-measure-Dirichlet spaces like fractals (see Conjectures~\ref{Sierpinski-conj} and~\ref{conj:SGlapbyrandgraphs} below).  They might also be combined with suitable theorems on the limiting distribution of i.i.d. points in metric spaces, such as those in~\cite{koltchinskii2000random} to obtain results for randomly sampled points.

\subsection{Optimal choice of $\epsilon$}

In the above discussion there are various choices for $\epsilon=\epsilon(n)$.  In the work of Belkin and Niyogi~\cite{belkin2008towards} one has for i.i.d.\ sample points on a manifold
\begin{equation*}
\epsilon(n) \asymp n^{- \frac1{d+2+\alpha}}\qquad \text{ for some } \alpha>0,
\end{equation*}
while in the same setting~\cite{gine2006empirical} require bounds 
 \begin{equation}
 n^{- \frac1{d+2}}\ll  \epsilon(n) \ll n^{- \frac1{d+4}}
 \end{equation} 
with high probability, if one is to obtain pointwise estimates on the quantity~\eqref{eqn:quantitytoestimate}.

 Similar questions can be asked in nonrandom situations.    For example, the Taylor estimate on a compact Riemannian manifold gives (see~\cite{gine2006empirical})
\begin{equation}
 |\mathcal L f- \mathcal L_{\epsilon} f| =O(\epsilon)\text{ as }\epsilon\to0.
 \end{equation}
Some of the authors of the present work have a corresponding result for weighted Riemannian manifolds with boundary, obtaining the same estimates for a limit Laplacian $\mathcal L$ that is an elliptic operator with boundary terms  (to be   published in \cite{ART}).  

Also in the non-random setting on a compact manifold, it is explained in~\cite{burago2015graph} that using Voroni cells to achieve a partition into equal measure subsets and suitably choosing vertices one can define a weighted Lapalcian such that  with $\epsilon(n)\asymp n^{-\frac1{2d}}$ one has approximation of the low eigenvalues and associated eigenprojectors of the Laplace-Beltrami operator by the graph Laplacian with error also of size $n^{-\frac1{2d}}$. 

\section {Uniform approximation 
	of averaging Laplacians on metric spaces}\label{sec:graphtoaveraging}

In this section we consider the uniform approximation $$\mathcal L_{\epsilon,n}\xrightarrow[n\to\infty]{}\mathcal L_{\epsilon}$$ of averaging Laplacians on metric spaces. 
It is often fairly easy to construct a weighted graph Laplacian that converges to a specified averaging Laplacian,  at least when it is applied to functions with some regularity.  The following two results give an simple construction in the special case that one knows how to partition the space into a union of sets of equal, and small, measure.  Note that this is analogous to the decomposition considered in~\cite{burago2015graph} and achieved on manifolds using Voroni cells; other decompositions with this property are also discussed there.

\begin{thm} \label{thm:discretelapbound}
Let $(X,\mathcal{F},\mu)$ be a measure space and fix $n\in\mathbb{N}$.  Suppose that for each $x\in X$  we have $N(x) \subset X$  a subset containing $x$ with $\mu(N(x))<\infty$ and assume it is possible to measure-theoretically equipartition $N(x)$ as $N(x)=\bigcup_{j=1}^n N_{j}(x)$ with $\mu(N_j(x))=\frac1n\mu(N(x))$ for each $j$.  Then choose  $x_j\in N_j(x)$ for $j=1,\dotsc,n$.
For any function $f$, Define
\begin{align*}
	L_nf(x)&= \frac{1}{n}\sum_{j=1}^n(f(x_j)-f(x))=\Bigl(\frac{1}{n}\sum_{j=1}^nf(x_j)\Bigr) -f(x),\text{ and} \\
	L_{N(x)}f(x)&=\frac{1}{\mu(N(x))}\int_{N(x)}(f(y)-f(x))d\mu(y).
\end{align*}
Then $\bigl|L_nf(x)-L_{N(x)}f(x)\bigr| \leq \sup_{j}\sup_{y\in N_j(x) } |f(x_j)-f(y)|$.
\end{thm}

\begin{proof}
Since $\mu(N(x))=n\mu(N_j(x))$ for each $j=1,\dotsc,n$ and the equipartition condition ensures $\mu(N_j(x)\cap N_k(x))=0$ when $j\neq k$
	\begin{align*}
	\bigl| L_nf(x)-L_{N(x)}f(x)\bigr|  &= \Bigl| \frac1n \sum_{j=1}^nf(x_j)-\frac{1}{\mu(N(x))}\int_{N(x)}f(y) d\mu(y) \Bigr|\\
	&=\frac1n \Bigl| \sum_{j=1}^n f(x_j)-\sum_{j=1}^n \frac{1}{\mu(N_j(x))}\int_{N_j(x)}f(y) d\mu(y) \Bigr| \\
	&\leq \frac1n\sum_{j=1}^n \frac{1}{\mu(N_j(x))}\int_{N_j(x)}\bigl| f(x)-f(y)\bigr| d\mu(y)\\
	& \leq \frac1n\sum_{j=1}^n\sup_{y\in N_j(x)}\bigl| f(x_j)-f(y)\bigr|\\
	& \leq \sup_j\sup_{y\in N_j(x)}\bigl| f(x_j)-f(y)\bigr| \qedhere
	\end{align*}  
\end{proof}

It is evident that the preceding result is only useful if we know there is a decomposition $\{N_j(x)\}$ such that $|f(x_j)-f(y)|$ is small on each $N_j(x)$ whenever $f$ is a function in the relevant function class.  Since we are mostly interested in approximating eigenfunctions and these typically have some regularity that can be expressed using a H\"older or Besov-Lipschitz norm, it is natural to think that we should try to make the sets $N_j(x)$  have small diameter. One reason we stated the preceding theorem is that a measure theoretic decomposition of the preceding type using small diameter sets is readily achieved on self-similar spaces, like the unit interval, the unit square, and the Sierpinski gasket studied in Section~\ref{fractalsection} below.

The property that one can decompose sets into equal small measure pieces is typical of self-similar spaces.  Let us say a metric measure space $(X,d,\mu)$ admits a measure-theoretic equipartition of diameter $\sigma>0$ if we can write $X=\bigcup_\eta X_\eta$ for measurable sets $X_\eta$ each having diameter at most $\sigma$ and satisfying $\mu(X_\eta)=\mu(X_{\eta'})$ and $\mu(X_\eta\cap X_{\eta'})=0$ for all $\eta\neq\eta'$.  Also say $X$ admits measure-theoretic equipartitions of arbitrarily small diameter if there is a sequence $\sigma_j\to0^+$ so that $X$ admits a measure-theoretic equipartition of diameter $\sigma_j$ for each $j$.

\begin{cor}\label{cor:graphlapapproxavlapifmeasurethpartitions}
Suppose $(X,d,\mu)$ is connected, of diameter at least $1$, and admits measure-theoretic equipartitions of arbitrarily small diameter.  Fix $\epsilon>0$ and $n\in\mathbb{N}$. Following Definition~\ref{def:avlap} one can define an averaging Laplacian operator $\mathcal{L}_\epsilon f(x)$ over sets of diameter at most $\epsilon$ and then a graph Laplacian operator $\mathcal{L}_{\epsilon,n}$ such that 
\begin{equation*}
\bigl|\mathcal{L}_\epsilon f(x) - \mathcal{L}_{\epsilon,n}f(x) \bigr|
\leq \sup_x \sup_{y\in B(x,\frac\epsilon n)} |f(x)-f(y)|
\end{equation*}
In particular, if $f$ is H\"older with exponent $\alpha$ then $\bigl|\mathcal{L}_\epsilon f(x) - \mathcal{L}_{\epsilon,n}f(x) \bigr|\leq C \bigl(\frac\epsilon n\bigr)^\alpha$.
\end{cor}
\begin{proof}
Fix $\delta>1$.  By assumption there is a measure-theoretic equipartition $X=\bigcup X_\eta$ in which all $X_\eta$ have diameter at most $\frac\epsilon {\delta n}$.  For each $\eta$, fix $x_\eta\in X_\eta$, arranging that $x_\eta\neq x_{\eta'}$ if $\eta\neq\eta'$.

Since  $X$ is connected and diameter at least $1$ there is $y_t\in X$ with $d(x,y_t)=t$ if $t\leq \epsilon<1$ (as otherwise $B(x,t)$ and the complement of its closure is a disconnection).  Take the points $y_{\frac{j\epsilon}n}$ with $0\leq j\leq n-1$.  Each is in some $X_\eta$, and since any two are separated by at least distance $\frac\epsilon n$ and $X_\eta$ has diameter $\frac\epsilon{\delta n}$, no two are in the same $X_\eta$.  Moreover, each of these $X_\eta$ is contained in $B(x,\epsilon)$.  It follows that we can find a set $S(x)=\{\eta_1,\dotsc, \eta_n\}$ containing exactly $n$ points so all $X_{\eta_j}$ are distinct and inside $B(x,\epsilon)$. 
 
We let $N(x)=\bigcup_{\eta\in S(x)} X_{\eta}$.  Then $N(x)$ is diameter at most $\epsilon$ and has positive measure, so Definition~\ref{def:avlap} gives simply $\mathcal{L}_\epsilon f(x)=L_{N(x)}f(x)$, the averaging operator of  Theorem~\ref{thm:discretelapbound}.  However, by construction $N(x)$ can also be measure-theoretically equipartitioned as in Theorem~\ref{thm:discretelapbound}, so we can define $\mathcal{L}_{\epsilon,n}f(x)=\frac1n\sum_{\eta\in S(x)} f(x_\eta)-f(x)=L_nf(x)$ in the notation of the theorem.  Notice that this is not the same as the graph Laplacian in Definition~\ref{def:graphlaplacians} because $\{x_\eta:\eta\in S(x)\}\subset \{\eta:d(x,x_\eta)\leq\epsilon\}$ but the inclusion may be strict.

In applying Theorem~\ref{thm:discretelapbound} we see that the bound is by the oscillation of $f$ over partition sets.  In this case, these are precisely the sets $X_\eta$ and have diameter not exceeding $\frac\epsilon {\delta n}$. It follows that we can replace it by the oscillation of $f$ over open balls of radius $\frac\epsilon n$, which gives the stated result.
\end{proof}

In metric measure spaces with doubling properties (of balls or of measures) there are partitioning results (eg the ``dyadic cubes'' of Christ~\cite{christ1990b}) which decompose the space into disjoint sets with controlled diameter and measure, but the measures of sets in the partition are not exactly the same and therefore the above argument cannot be applied to these partitions.  It would be useful to have a more flexible argument that could apply to this more general situation, perhaps allowing the possibility that the graph Laplacians converge to an average with respect to a kernel rather than simply an average over a specified set $N(x)$.  If the latter is needed, we would also need some control on the behavior of the kernel.  As we shall see, the types of kernels that are useful are those that integrate to zero against the first order terms of a ``Taylor expansion'' in the space, when such an expansion can be suitably defined.

We remark that the argument given here is far from the only way to analyze the convergence of graph Laplacians to averaging operators.  For a different approach considering the convergence of spectra of graph Laplacians to the spectra of a certain type of averaging Laplacian see~\cite{burago2015graph,burago2019spectral}.

\section{Random and nonrandom sampling from $[-1,1]$}

Although the interval $[-1,1]$ with Lebesgue measure is an elementary example, to the best of our knowledge the behavior of the Laplacian eigenmap corresponding to (random or non-random) sample points on this interval cannot be immediately deduced from  general results in the literature. This is because the general results we are aware of for manifolds deal with the  Laplace-Beltrami operator on a compact Riemannian manifold without boundary, or an open subset of such a    manifold that is separated from the boundary by a distance that is bounded below, neither of which is the case for $[-1,1]$.  
Moreover, $[-1,1]$ fairly immediately demonstrates the importance of boundary conditions for the problem at hand: we discuss below the fact that with Neumann or Dirichlet boundary conditions the Laplacian eigenmap takes $[-1,1]$ to an arc, whereas with period boundary conditions the image under the eigenmap is a topological circle.  In what follows, we use the interval as an elementary testbed for beginning to understand the behavior of Laplacian eigenmaps on a manifold with boundary.

On this example we consider the following questions from Section~\ref{ssec:questionsandresults}.  Can one approximate a Laplacian $\mathcal{L}$ by $\epsilon^{-\beta}\mathcal{L}_{\epsilon}$ on a natural class of functions? If so, what is the class of functions and can we prove a rate of convergence for this class?   Can one approximate $\mathcal{L}_{\epsilon}$ by graph Laplacians corresponding to uniformly spaced points, or to randomly spaced points?  If so, what is the rate of convergence and in the random cases how does it depend on the distribution of the random points?

\subsection{Laplacians}\label{ssec:IntervalLaplacians}
It is possible to characterize all Dirichlet forms on $[-1,1]$ and thus all Laplacians that could be considered on $[-1,1]$; for some discussion of this see~\cite[Examples~1.2.2 and~1.3.1 and~1.3.2]{FOT}.  Here we limit ourselves to some simple subclasses of those Laplacians which arise as self-adjoint extensions of  the weak second derivative $A=-\frac{d^2}{dx^2}$ with domain $C^2_0(-1,1)$, which is a closed symmetric operator on $L^2(dx)$. $A$ is invariant under complex conjugation, so has self-adjoint extensions by von Neumann's theorem (\cite{ReedSimon} Theorem~X.3). These may be found by elementary means. If $B$ is a symmetric extension of $A$ and $f\in\dom(B)$, $g\in\dom(B^*)$ then 
\begin{align}
0=\langle  Bf, g \rangle -  \langle f, B^\ast g \rangle =\int_{-1}^1 -f'' g  +  f g''
&=\int_{-1}^1 \frac{d}{dx}( g f' -f g') \notag\\
&= ( g f' -fg')(1)- ( g f' -f g')(-1) \label{eqn:selfadjextensionsof2ndderiv}
\end{align}
and we see that the self adjoint extensions of $A$ are those with domain the absolutely continous functions on $[-1,1]$ (written $f\in\AC[-1,1]$) subject to boundary conditions for $f$ having the property that~\ref{eqn:selfadjextensionsof2ndderiv} implies these same boundary conditions are true for $g$. The details are similar to those treated in Examples 1 and 2 in Section~X.1 of~\cite{ReedSimon}, so we omit them.  Instead we give some classical examples that illustrate this idea.

\begin{itemize}
\item
The Dirichlet Laplacian is $\frac{d^2}{dx^2}$ with domain $\{f\in\AC[-1,1]: f(-1)=f(1)=0\}$. Rewriting~\ref{eqn:selfadjextensionsof2ndderiv} in the form $fg'(1)-fg'(-1) = gf'(1)-gf'(-1)$ and noting that $f'(-1)$, $f'(1)$ can take any value we see $ f(-1)=f(1)=0$ implies $g(-1)=g(1)=0$ so this is a self-adjoint extension of $A$.
\item By similar reasoning the Neumann Laplacian  $\frac{d^2}{dx^2}$ with domain $\{f\in\AC[-1,1]: f'(-1)=f'(1)=0\}$ is self-adjoint because the boundary condition and~\ref{eqn:selfadjextensionsof2ndderiv} readily shows $g'(-1)=g'(1)=0$.
\item The preceding are boundary conditions in which the values of $f,f'$ at the endpoint $-1$ are unrelated to those at the endpoint $1$. The most general such conditions have $f'(-1)=af(-1)$ and $f'(1)=bf(1)$ for some $a,b\in\mathbb{R}\cup\{\infty\}$ (the Neumann case is $a=b=0$ and the Dirichlet case is $a=b=\infty$) and lead to a Robin Laplacian  $\frac{d^2}{dx^2}$ with domain $\{f\in\AC[-1,1]: f'(-1)=af(-1),\,f'(1)=bf(1) \}$.
\item One can also obtain self-adjoint operators for which the boundary conditions at $-1$ and $1$ are related. Perhaps the most familiar arises from periodic boundary conditions to obtain $\frac{d^2}{dx^2}$ with domain $\{f\in\AC[-1,1]: f(-1)=f(1),\, f'(-1)=f'(1)\}$, and are equivalent to considering the second derivative on a circle.  However, other choices are possible; the reader may care to check that $\frac{d^2}{dx^2}$ can be made self-adjoint with domain  $\{f\in\AC[-1,1]:f(-1)=-f(1)\, f'(-1)=-f'(1)\}$ or with domain $\{f\in\AC[-1,1]:f(-1)=f(1)=0\, f'(-1)=-f'(1)\}$, for example. 
\end{itemize}
 
For any specified set of boundary values it is easy enough to compute the eigenvalues and eigenfunctions.  The solutions of $\phi''(x)=-\lambda^2 \phi(x)$ are linear combinations $A e^{ i\lambda x}+Be^{-i\lambda x}$.  For example, assuming Robin boundary conditions $\phi'(-1)=a\phi(-1)$, $\phi'(1)=b\phi(1)$ hold for real $a,b$ gives a linear system
\begin{align*}
i\lambda (A- Be^{2i\lambda})&=a(A+Be^{2i\lambda})\\
i\lambda(A-Be^{-2i\lambda})&=b(A+Be^{-2i\lambda})
\end{align*}
so that the eigenvalues are roots of $e^{4i\lambda}=\frac{(a-i\lambda)(b+i\lambda)}{(a+i\lambda)(b-i\lambda)}$, or equivalently $\lambda=\arg(b+i\lambda)-\arg(a+i\lambda)$, and one can solve for $A$ and $B$ to obtain the corresponding eigenfunction.  The special case $a=b=0$ (the Neumann Laplacian) has $\lambda=\frac 12k \pi$ hence eigenvalues $\bigl(\frac 12 k\pi \bigr)^2$ for $k\in\mathbb{N}\cup\{0\}$ and corresponding (real-valued) eigenfunctions $\cos \frac12(k \pi x)$ if $k$ is even and $\sin \frac12(k \pi  x)$ for $k$ odd.  The Dirichlet Laplacian has the same eigenvalues except that $0$ is not an eigenvalue, so the spectrum is $\{\bigl(\frac 12 k\pi \bigr)^2:k\in\mathbb{N}\}$,  but the (real-valued) eigenfunctions are $\sin \frac12(k \pi x)$ for $k$ even and $\cos \frac12(k \pi x)$ for $k$ odd.

Boundary conditions that involve data at both endpoints maybe treated in a similar way, though the linear system is different.  For example, the periodic boundary conditions described above lead to the system
\begin{align*}
Ae^{-i\lambda} + Be^{i\lambda}&=Ae^{i\lambda} + Be^{-i\lambda}\\
i\lambda(Ae^{-i\lambda} - Be^{i\lambda}) &= i\lambda( Ae^{i\lambda} - Be^{-i\lambda})
\end{align*}
which has non-trivial solutions for $\lambda=\pi k$, $k\in\mathbb{N}\cup\{0\}$. Thus the are eigenvalues $(\pi k)^2$, with a two-dimensional eigenspace  spanned by (real-valued) eigenfunctions $\sin(k\pi x)$ and $\cos(k\pi x)$ for each $k\geq 1$ and a one-dimensional eigenspace with constant eigenfunction when $k=0$.

From the preceding we can easily write the Laplacian eigenmap using the two non-constant eigenfunctions with smallest eigenvalue in our basic cases.  For the Neumann Laplacian these eigenfunctions (corresponding to $k=1,2$ are $\sin \frac12(\pi x)$ and  $\cos (\pi x)=1-2\sin^2\frac12(\pi x)$, so the eigenmap takes $[-1,1]$ to an arc on the parabola $\{(y,1-4y^2)\}\subset\mathbb{R}^2$.  For the Dirichlet Laplacian we again have $k=1,2$, now with eigenfunctions $\cos\frac12(\pi x)$ and $\sin(\pi x)=2(\sin\frac12\pi x)(\cos\frac12\pi x)$, so the eigenmap takes $[-1,1]$ to an arc of the curve $\{(y,2y\sqrt{1-y^2}\}\subset\mathbb{R}^2$.  For the periodic Laplacian we have only $k=1$ because the eigenspace is two dimensional; the eigenfunctions are $\sin \pi x$ and $\cos \pi x$ and the image of the interval is a circle.

\subsection{Averaging Laplacians and their convergence}

For $\epsilon>0$ we use the notation $\tilde{B}(x,\epsilon)=(x-\epsilon,x+\epsilon)\cap[-1,1]$ and write $|\tilde{B}(x,\epsilon)|$ for its Lebesgue measure, then set 
\begin{equation*}
	\mathcal{L}_\epsilon f(x)=|\tilde{B}(x,\epsilon)|^{-1}\int_{\tilde{B}(x,\epsilon)} f(y)-f(x) dy
\end{equation*}
for $x\in[-1,1]$.

\begin{thm}\label{thm:Lepstof''}
If $f'$ is absolutely continuous on $[-1,1]$ and $f'(-1)=f'(1)=0$ then
\begin{equation*}
\lim_{\epsilon\to0} \epsilon^{-2}\mathcal{L}_\epsilon f(x)=\frac16 f''(x)
\end{equation*}
at a.e. $x\in[-1,1]$. If $f''$ is continuous this limit is valid everywhere, and if $f''$ is H\"older with exponent $\alpha\in(0,1]$ then the convergence is uniform with $|\epsilon^{-2}\mathcal{L}_\epsilon f(x)-\frac16f''(x)|\leq C\epsilon^\alpha$ for some $C>0$.
\end{thm}
\begin{proof}
For $x\in(-1,1)$ it is apparent that once $\epsilon$ is less than the distance from $x$ to $\{-1,1\}$ we have $\tilde{B}(x,\epsilon)=B(x,\epsilon)$. Taylor's theorem gives 
\begin{equation*}
	f(y)-f(x)=f'(x)(y-x)+\int_x^y f''(t)(y-t)\, dt.
\end{equation*}
When we integrate over $B(x,\epsilon)$ the linear term cancels by symmetry and we may translate the coordinates to $x$ so as to obtain
\begin{align}
\lefteqn{\int_{B(x,\epsilon)} f(y)-f(x)dy}\quad& \notag\\
&=  \int_{x-\epsilon}^{x+\epsilon} \int_x^y f''(t)(y-t)\,dt\,dy \notag\\
&=  \int_{-\epsilon}^\epsilon \int_0^y f''(x+t)(y-t)\,dt\,dy \notag\\
&= \int_{-\epsilon}^0 f''(x+t)\int_{-\epsilon}^t (y-t)\,dy\, dt+ \int_0^\epsilon f''(x+t) \int_t^\epsilon  (y-t)\,dy\, dt \notag\\
&= \frac12 \int_{-\epsilon}^0f''(x+t) (\epsilon+t)^2\,dt + \frac12\int_0^\epsilon f''(x+t)(\epsilon-t)^2\,dt \label{eqn:avlapllemmaeqn1}
\end{align}
Evaluating the same integrals with $f''(x+t)$ replaced by $f''(x)$ yields $\frac{\epsilon^3}{3}f''(x)$, while dividing by $2\epsilon$ gives $\mathcal{L}_\epsilon$, so for $\epsilon$ less than the distance from $x$ to $\{-1,1\}$
\begin{align} 
\lefteqn{\Bigl| \epsilon^{-2}\mathcal{L}_\epsilon f (x) -\frac16 f''(x)\Bigr|}&\quad \notag\\
&=\biggl| \frac1{4\epsilon^3}\int_0^\epsilon \bigl( f''(x-\epsilon+t)-f''(x)) t^2\,dt + \frac1{4\epsilon^3}\int_{-\epsilon}^0 \bigl( f''(x+\epsilon+t)-f''(x)\bigr) t^2\,dt \biggr| \notag\\
&\leq \frac1{4\epsilon} \int_{x-\epsilon}^{x+\epsilon}|f''(y)-f''(x)|\,dy. \label{eqn:avlapllemmaeqn2}
\end{align}
If $f''$ is continuous this is easily seen to have limit zero as $\epsilon\to0$, and the H\"older estimate is also elementary.

The two boundary points $x\in\{-1,1\}$ are symmetrical, so we consider only $x=-1$. In this case $\tilde{B}(-1,\epsilon)=[-1,-1+\epsilon)$ for any $0<\epsilon<2$.  We again use Taylor's theorem but this time find
\begin{align*}
\lefteqn{\int_{\tilde{B}(-1,\epsilon)} f(y)-f(x)dy}\quad&\\
&= f'(-1)\int_{-1}^{-1+\epsilon} (y+1)\, dy + \frac12 \int_{-1}^{-1+\epsilon} \int_{-1}^y f''(t)(y-t)\,dt\,dy \\
&= \frac12\epsilon^2 f'(-1) + \frac14 \int_0^{\epsilon} f''(-1+t)(\epsilon-t)^2\,dt.
\end{align*}
Dividing by $\epsilon^2|\tilde{B}(-1,\epsilon)|=\epsilon^3$ to get $\epsilon^{-2}\mathcal{L}_\epsilon$ will give a finite limit only when $f'(-1)=0$, in which case the same argument as before gives that the limit is $\frac16f''(-1)$ if $f''$ is continuous on $[-1,1]$ and the estimate in the H\"older case is valid.

Finally, if we have only $f'\in\AC[-1,1]$ then $f'$ is continuous so the conditions $f'(-1)=f'(1)=0$ make sense and our reasoning via Taylor's theorem is justified because $f''\in L^1[-1,1]$.  In this case the a.e. limit is proved by the Lebesgue differentiation theorem.
\end{proof}

It is worth noting that a similar argument applies to $\mathcal{L}_{K,\epsilon}$ defined using averaging with respect to a compactly supported kernel.  To be specific, suppose we have an even function $k(s)\geq0$ supported in $[-1,1]$ so $\int k(s)ds=1$. For $x\in[-1,1]$ define $K_\epsilon(x,y)$ to be $k(\epsilon^{-1}d(x,y))$ but truncated to $y\in[-1,1]$ and normalized to have integral $1$. Then our definitions are
\begin{align*}
K_\epsilon(x,y) &= \frac{k(\epsilon^{-1}d(x,y))\mathbf{1}_{[-1,1]}(y)} {\int_{-1}^1 k(\epsilon^{-1}d(x,y))\,dy}, \text{ and}\\
\mathcal{L}_{K,\epsilon} &= \int (f(y)-f(x)) K_\epsilon(x,y) dy.
\end{align*}
As in the preceding we see that if $x\in(-1,1)$ and $\epsilon$ is less than the distance from $x$ to $\{-1,1\}$ then this simplifies to $K_\epsilon(x,y)=\epsilon^{-1}k(\epsilon^{-1}d(x,y))$.  We can repeat the argument that leads to~\ref{eqn:avlapllemmaeqn1}; cancellation of the linear term now uses that $k$ is an even function, and the resulting expression is
\begin{equation*}
 \int_{-\epsilon}^0 f''(x+t)\int_{-\epsilon-t}^0 k(\epsilon^{-1}|y|) y\,\frac{dy}\epsilon dt + \int_0^\epsilon f''(x+t) \int_0^{\epsilon-t} k(\epsilon^{-1}|y|) y\,\frac{dy}\epsilon \, dt.
\end{equation*}
If $M=\frac12\int_{-1}^1 k(y)y^2\,dy$ then subtracting $M\epsilon^2 f''(x)$ and reorganizing the integration we obtain an estimate analogous to~\ref{eqn:avlapllemmaeqn2}
\begin{equation*}
\Bigl| \epsilon^{-2}\mathcal{L}_{K,\epsilon}f(x) - M f''(x) \Bigr| \leq \int K_\epsilon(x,y) |f''(y)-f''(x)|dy
\end{equation*}
and conclude that this converges to zero as $\epsilon\to0$.  The argument at the points $\{-1,1\}$ is also largely unchanged, so in the case that $f'$ vanishes at $\pm1$ we find $\epsilon^{-2}\mathcal{L}_{K,\epsilon}f(x)$ converges to $Mf''(x)$ with $M$ being the second moment of $k$.

\subsection{Graph Laplacians for equally spaced points}

Let $S_n=\{x_1,\dotsc, x_n\}$ be equally spaced points in $[-1,1]$, so that $x_j=-1+\frac{2(j-1)}{n-1}$.  For $x\in[-1,1]$ form the graph Laplacian $\mathcal{L}_{n,\epsilon}$ as in Definition~\ref{def:graphlaplacians}.  Recall that in doing so at a point $x$ we take $x_0=x$ and form the graph on $S_n\cup\{x_0\}$ with edge weights between $x_i\neq x_j$ being $w_{i,j}=1$ if $|x_i-x_j|\leq \epsilon$ and $w_{i,j}=0$ otherwise.  Here we study the convergence of $\mathcal{L}_{n,\epsilon}f(x)$ as $n\to\infty$.

\begin{thm}\label{thm:graphtoaveragingoninterval}
For $f$ on $[-1,1]$ the difference between the averaged Laplacian $\mathcal{L}_\epsilon$ and the graph Laplacian $\mathcal{L}_{n,\epsilon}$ is dominated by the oscillation of the function at scale $\frac2n$: precisely, provided $n>\frac1\epsilon$, 
\begin{equation*}
|\mathcal{L}_\epsilon f(x) - \mathcal{L}_{n,\epsilon}f(x) | \leq  \sup_{|y-z|<\frac2n}|f(y)-f(z)|
\end{equation*}
If $f$ is H\"older with exponent $\alpha$ then $|\mathcal{L}_\epsilon f(x) - \mathcal{L}_{n,\epsilon}f(x) | \leq C n^{-\alpha}$.
\end{thm}
\begin{proof}
We use Theorem~\ref{thm:discretelapbound}.  Fix $x$ and $\epsilon>0$. Let $N(x)=[x-\epsilon,x+\epsilon)$, so $\mathcal{L}_\epsilon f(x) = L_{N(x)}f(x)$ in the notation of that theorem. For $n>\frac1\epsilon$ fix $k\in\mathbb{N}$ such that $k\leq 2n\epsilon<k+1$, so that $N(x)$ contains exactly $k$ of the points from $S_n$.  Now subdivide $N(x)$ into $k$ equal intervals of length $\frac{2\epsilon}k$, with endpoints $x-\epsilon+\frac{2\epsilon j}k$ for $0\leq j\leq k$ and label them $N_j(x)$.  It is easy to verify that there is a unique $x_j\in S_n\cap N_j(x)$, so we label it $x_{j}$ and record from Theorem~\ref{thm:discretelapbound} that
\begin{equation*}
	|L_{N(x)}f(x) - L_kf(x)|\leq \sup_{x_{j}\in S_n\cap N(x)} \sup_{y\in N_j(x)} |f(x_{j})-f(y)|.
	\end{equation*}
Using the fact that each $N_j(x)$ is length $\frac{2\epsilon}k<\frac2n$ we can replace the right side by an oscillation bound over intervals of length $\frac 2n$.  The result is independent of $x$ and $k$:
\begin{equation*}\label{eqn:graphtoaveragingoninterval1}
	|L_{N(x)}f(x) - L_kf(x)|\leq \sup_{|y-z|<\frac2n}|f(y)-f(z)|.
	\end{equation*}

Now from the proof of  Theorem~\ref{thm:discretelapbound}, 
\begin{equation*}
 L_kf(x) = \frac1k \sum_{j=1}^k (f(x_{j})-f(x) = \mathcal{L}_{n,\epsilon}f(x)
\end{equation*}
because the points $x_{j}$ are precisely those in $S_n\cap[x-\epsilon,x+\epsilon)$, so we have the graph with edge weight $\frac1k$ between $x$ and the $k$ points of $S_n$ that are within distance $\epsilon$ and edge weight zero otherwise.  Notice that $k$ is the degree of $x$ in the graph.

Since also $\mathcal{L}_\epsilon f(x) = L_{N(x)}f(x)$ we may rewrite our estimate~\eqref{thm:graphtoaveragingoninterval} as 
\begin{equation*}
	|\mathcal{L}_\epsilon f(x) - \mathcal{L}_{n,\epsilon}f(x) | \leq \sup_{|y-z|<\frac2n}|f(y)-f(z)|.
	\end{equation*}
from which the H\"older bound follows.
\end{proof}

\subsection{Rate of convergence of graph Laplacians using uniformly spaced sample points to the Laplacian operator}

Combining the results from the previous two sections we obtain an estimate of the rate of convergence of graph Laplacians for uniformly spaced points to the second derivative. This estimate is not optimal, but it serves to illustrate the situation with random approximations.

\begin{pr}\label{thm:rateLntoLoninterval}
If $f:[-1,1]\to\mathbb{R}$ satisfies $f'(-1)=f'(1)=0$ and $f''$ is H\"older continuous with exponent $\alpha$ then  
there is $\epsilon(n)$ and a constant $C$ so
\begin{equation}\label{eqn:rateLntoLoninterval1main}
\bigl| \frac16 f''(x)-\epsilon(n)^{-2}\mathcal{L}_{\epsilon,n} f(x) \bigr|
\leq C n^{-\frac \alpha{\alpha+2}}.
\end{equation}
\end{pr}
\begin{proof}
Since $f$ has continuous derivative there is $C_1$ so $|f(y)-f(z)|\leq C_1|y-z|\leq \frac{2C_1}{n}$ if $|y-z|<\frac2n$.  We use this in the bound for Theorem~\ref{thm:graphtoaveragingoninterval}.   By hypothesis there is also $C_2$ so $|f''(y)-f''(z)|\leq C_2|y-z|^\alpha$, which we use in the estimate of Theorem~\ref{thm:Lepstof''}. Combining these estimates gives
\begin{align}
\bigl| \frac16 f''(x)-\epsilon^{-2}\mathcal{L}_{\epsilon,n} f(x) \bigr|
&\leq \bigl| \mathcal{L}f(x)-\epsilon^{-2}\mathcal{L}_{\epsilon} f(x) \bigr| + \epsilon^{-2} \bigl| \mathcal{L}_\epsilon f(x)-\mathcal{L}_{\epsilon,n} f(x) \bigr| \notag\\
&\leq C_2 \epsilon^\alpha + \frac{2C_1}{n\epsilon^2}. \label{eqn:thm:rateLntoLoninterval1}
\end{align}
Minimizing over $\epsilon$ for fixed $n$ we set $\epsilon(n)=\Bigl(\frac{4C_1}{C_2\alpha n}\Bigr)^{\frac1{\alpha+2}}$ and obtain~\eqref{eqn:rateLntoLoninterval1main} for a constant $C$ depending on $C_1$, $C_2$ and $\alpha$.
\end{proof}

In particular, if $f'''$ is continuous we find that approximating the second derivative by equally spaced graph Laplacian points gives an error at most $Cn^{-1/3}$.  If this were the true error for approximating eigenfunctions then we might expect a Laplacian eigenmap to do a poor job of reflecting the geometry of a data set that lies on a manifold, unless there are a great many data points.  However, in the case of uniformly spaced points on an interval we can explicitly compute the eigenfunctions of the graph Laplacian and determine how close they are to the actual eigenfunctions.  Doing so shows that the error is much smaller than the preceding analysis suggests.  We give the full computation only for an illustrative special case: $\epsilon=\frac2{n-1}$. Note that the preceding error analysis, as in~\eqref{eqn:thm:rateLntoLoninterval1}, produces completely uninteresting bounds in this situation.  However, it is evident from the following that something different occurs.

\begin{pr}\label{prop:efnsofgraphlaplacianonI}
For the graph Laplacian $\mathcal{L}_{\frac2{n-1},n}$ at points of $S_n=\{-1+\frac{2(j-1)}{n-1}:1\leq j\leq n\}$, the eigenvalues have the form $\cos\bigl(\frac {k\pi}{n-1}\bigr) -1$ with eigenfunction $\cos\bigl(\frac12k\pi x\bigr)$ for $k$ even and $\sin\bigl(\frac12 k\pi x\bigr)$ for $k$ odd, where $0\leq k\leq n-1$ is an integer.  These are exactly the eigenfunctions of the Neumann Laplacian on $[-1,1]$. 
\end{pr}
\begin{proof}
At $x_j$, $1<j<n$ we have $\mathcal{L}_{\frac2{n-1},n}f(x_j)=\frac12\bigl(f(x_{j+1})+f(x_{j-1})-2f(x_j)\bigr)$.  With $f(x)=e^{\pm i\lambda x}$ we find 
\begin{equation*}
	\mathcal{L}_{\frac2{n-1},n}f(x_j)=\frac12 f(x_j) \Bigl( e^{\frac{i2\lambda}{n-1}} +  e^{\frac{-i2\lambda}{n-1}}-2\Bigr)
	=  f(x_j) \Bigl(\cos\bigl(\frac{2\lambda}{n-1}\bigr)-1\Bigr).
	\end{equation*}
Now taking $f=Ae^{i\lambda x} +Be^{-i\lambda x}$ we compute at $x_1=-1$ and $x_n=1$ that
\begin{align*}
\mathcal{L}_{\frac2{n-1},n}f(x_1)
 &= f(x_2)-f(x_1)
 = Ae^{-i\lambda}\bigl(e^{\frac{i2\lambda}{n-1}}-1\bigr)+ Be^{i\lambda}\bigl(e^{\frac{-i2\lambda}{n-1}}-1\bigr) \\
\mathcal{L}_{\frac2{n-1},n}f(x_n)
&= f(x_{n-1})-f(x_n)
= Ae^{i\lambda} \bigl(e^{\frac{-i2\lambda}{n-1}}-1\bigr) + Be^{i\lambda}\bigl(e^{\frac{i2\lambda}{n-1}}-1\bigr)
\end{align*}
so the eigenvalue equations are
\begin{align*}
Ae^{-i\lambda}\Bigl( e^{\frac{i2\lambda}{n-1}}- \cos\bigl(\frac{2\lambda}{n-1}\bigr)\Bigr) 
+ Be^{i\lambda} \Bigl( e^{\frac{-i2\lambda}{n-1}}- \cos\bigl(\frac{2\lambda}{n-1}\bigr)\Bigr) 
&=0\\
Ae^{i\lambda}\Bigl( e^{\frac{-i2\lambda}{n-1}}- \cos\bigl(\frac{2\lambda}{n-1}\bigr)\Bigr) 
+ Be^{-i\lambda} \Bigl( e^{\frac{i2\lambda}{n-1}}- \cos\bigl(\frac{2\lambda}{n-1}\bigr)\Bigr)
&=0
\end{align*}
which is a linear system for $A$ and $B$ that can be rewritten as
\begin{equation*}
\sin\bigl(\frac{2\lambda}{n-1}\bigr)
 \begin{bmatrix} e^{-i\lambda} & -e^{i\lambda}\\ e^{i\lambda} & -e^{-i\lambda} \end{bmatrix}
\begin{bmatrix} A\\ B\end{bmatrix}
= \begin{bmatrix} 0\\0\end{bmatrix} 
\end{equation*}
and has nontrivial solutions exactly when $\sin(2\lambda)\sin\bigl(\frac{2\lambda}{n-1}\bigr)=0$. This gives $\lambda=\frac12 k\pi$ for $k\in\mathbb{Z}$ and it is easy to check that then $A=B$ if $k$ is even and $A=-B$ if $k$ is odd. The eigenfunctions are then $\cos \lambda x=\cos \frac12k\pi$ if $k$ is even and $\sin \lambda x=\sin\frac12 k\pi$ if $k$ is odd.  We need only consider $0\leq k\leq n-1$ by periodicity (in $k$) of the expression for the eigenvalue and of the eigenfunctions restricted to $S_n$.

The fact that these are also exactly the  eigenfunctions of the Neumann Laplacian was discussed at the beginning of Section~\ref{ssec:IntervalLaplacians}.
\end{proof}

There is also a version of the bound in~\eqref{eqn:thm:rateLntoLoninterval1} for the case of an eigenfunction with small eigenvalue.  Note that since $\epsilon=\frac2{n-1}$ we have $\epsilon^{-2}=\frac14(n-1)^2$, and we can compute only at points of $S_n$.  However $f$ varies only a small amount (bounded by a constant multiple of $\frac1n$) across the intervals between points of $S_n$.
\begin{cor}\label{cor:efnsofgraphlaplacianonI}
For $f$ an eigenfunction of the Neumann Laplacian $\mathcal{L}=\frac12\frac{d^2}{dx^2}$ on $[-1,1]$, or equivalently of $\mathcal{L}_{\frac 2{n-1},n}$ on $S_n$, associated to one of the first $k_0$ non-zero eigenvalues then there is $C$ so
\begin{equation*}
	\bigl| \mathcal{L}f(x_j) - \frac14(n-1)^2 \mathcal{L}_{\frac2{n-1},n} f \bigr| \leq Ck_0^4 (n-1)^{-2}
	\end{equation*} 
\end{cor}
\begin{proof}
Since $f$ is an eigenfunction of both operators it is sufficient to bound the difference between the eigenvalues.  The eigenvalue of the $k^\text{th}$ eigenfunction of $\mathcal{L}$ is $\frac12\bigl(\frac12 k\pi\bigr)^2$. The preceding theorem says the $k^\text{th}$ eigenvalue for $n^2\mathcal{L}_{\frac1n,n}$ is $ \frac14(n-1)^2 \Bigl(\cos\bigl(\frac{k\pi}{n-1}\bigr)-1 \Bigr)$. From the power series for cosine (and using $\frac{k\pi}{n-1}$) is bounded) we have
\begin{align*}
	\Bigl|\frac12 \bigl(\frac12 k\pi\bigr)^2 - \frac14(n-1)^2 \Bigl(\cos\bigl(\frac{k\pi}{n-1}\bigr)-1 \Bigr)\Bigr|
	&\leq C \Bigl( \frac{k^4}{(n-1)^2}\Bigr)\leq  Ck_0^4 (n-1)^{-2}
	\end{align*}
because $k\leq k_0$.
\end{proof}
It is important to note that in the preceding we have that the graph Laplacian is close to $\frac12\frac{d^2}{dx^2}$, not to the limit $\frac16\frac{d^2}{dx^2}$ found in Theorem~\ref{thm:Lepstof''}.  Moreover, this is only for the graph Laplacian on equally spaced sample points and is only interesting for low frequency eigenfunctions.  Nonetheless, it emphasizes the possibility that one might get quite an accurate eigenmap from a graph Laplacian $\mathcal{L}_{\epsilon,n}$ with $\epsilon$ much larger than would be required to apply Proposition~\ref{thm:rateLntoLoninterval}.  We return to this at the end of the next section.

\subsection{Average graph Laplacian for i.i.d.\ sample points}\label{ssec:iidlaponinterval}

We construct a random graph Laplacian as discussed in Definition~\ref{def:randomgraphLap}. To this end, fix a probabilty space $(\Omega,\mathbb{P})$ and take $n$ independent, identically distributed, random variables $x_j:\Omega\to [-1,1]$.  For $\omega\in\Omega$ let $S_n(\omega)=\{x_j(\omega):1\leq j\leq n\}$ and for $x\in[-1,1]$ define the graph Laplacian on $S_n(\omega)\cup\{x\}$ with edges between points separated by distance at most $\epsilon$ by setting $E_\epsilon(x)=\{x_j(\omega):|x-x_j(\omega)|\leq \epsilon\}$ and 
\begin{equation*}
	\mathcal{L}_{\epsilon,n}(\omega)f(x)
	= \frac1{\# E_\epsilon}\sum_{E_\epsilon} f(x_j(\omega))-f(x)
\end{equation*}
provided $\# E_\epsilon(x)\neq0$; in the case $\# E_\epsilon(x)=0$ we set $\mathcal{L}_{\epsilon,n}(\omega)f(x)=0$.

In order to understand $\mathcal{L}_{\epsilon,n}(\omega)$ it is useful to consider a simpler object defined as follows:
\begin{equation*}
	L_{\epsilon,n}(\omega) f(x) = \frac1{\epsilon n} \sum_{j=1}^n \mathbf{1}_{B(x,\epsilon)}(x_j) \bigl( f(x_j)-f(x)\bigr)
	\end{equation*}
which is similar to $\mathcal{L}_{\epsilon,n}$ but has the advantage that at each $x$ it is an average of i.i.d.\ random variables $\frac1\epsilon  \mathbf{1}_{B(x,\epsilon)}(x_j) \bigl( f(x_j)-f(x)\bigr)$ and therefore has convergence described by the central limit theorem as in Theorem~\ref{thm-monte-carlo} (see also Theorems \ref{thm2.4} and \ref{thm0.4}). 

\begin{thm}\label{thm:randomlaponinterval}
Suppose $f$ is such that $f'(1)=f'(-1)=0$ and $f''$ is H\"older continuous of exponent $\alpha$.  

If we choose $\epsilon>0$ and $n\to\infty$ then at each $x\in[-1,1]$
\begin{equation*}
\sqrt{ n\epsilon^3} \bigl( \epsilon^{-2} L_{\epsilon,n}(\omega) f(x) - \frac16 f''(x) + O(\epsilon^\alpha)\bigr)  \xrightarrow[\text{distr.}]{} \frac1{\sqrt{3}} |f'(x)|(1+O(\epsilon))  Z 
\end{equation*}
where $Z$ is a standard normal distribution $N(0,1)$ and the convergence is in distribution.  

If we choose $\epsilon(n)\to0$ so that $n\epsilon(n)^3\to\infty$ then at each $x\in[-1,1]$
\begin{equation*}
\sqrt{ n\epsilon(n)^3} \bigl( \epsilon(n)^{-2} L_{\epsilon,n}(\omega) f(x) - \frac16 f''(x) \bigr)  \xrightarrow[\text{distr.}]{} \frac1{\sqrt{3}} |f'(x)|  Z.
\end{equation*}

Minimizing on $\epsilon(n)$ 
in \eqref{eq-min-eps} we find there is a constant $C$ so
\begin{equation*}
\mathbb{E}\Bigl| \epsilon^{-2} L_{\epsilon,n}(\omega) f(x) - \frac16 f''(x)\Bigr|
\leq C n^{-\frac\alpha{3+2\alpha}}.
\end{equation*}
\end{thm}	
\begin{proof}
Compute
\begin{align*}
\lefteqn{\Var \Bigl( \epsilon^{-1} \mathbf{1}_{B(x,\epsilon)}(x_j) \bigl( f(x_j)-f(x)\bigr) \Bigr)}\quad&\\
&=\epsilon^{-2} \mathbf{E}\Bigl(  \mathbf{1}_{B(x,\epsilon)}(x_j) \bigl( f(x_j)-f(x)\bigr)^2 \Bigr)\\
&= \epsilon^{-2} \mathbf{E}\Bigl(  \mathbf{1}_{B(x,\epsilon)}(x_j) (f'(x))^2 (x_j-x)^2(1+O(\epsilon)) \Bigr)\\
&= \frac13 (f'(x))^2\epsilon(1+ O(\epsilon))
\end{align*}
and
\begin{equation*}
\mathbb{E}\Bigl(\epsilon^{-1} \mathbf{1}_{B(x,\epsilon)}(x_j) \bigl( f(x_j)-f(x)\bigr) \Bigr)
=
 \frac1{2\epsilon} \int_{B(x,\epsilon)} f(y)-f(x)\, dy
= \mathcal{L}_\epsilon f(x)
\end{equation*}
to see by substitution into the result of Theorem~\ref{thm-monte-carlo} and using $\sqrt{1+O(\epsilon)}=1+O(\epsilon)$ that
\begin{equation*} 
\frac { \sqrt{n} \bigl( L_{\epsilon,n}(\omega) f(x) - \mathcal{L}_\epsilon f(x) \bigr)} {  \frac1{\sqrt{3}} (f'(x))\sqrt{\epsilon}(1+O(\epsilon))  } \to Z
\end{equation*}
in distribution, where $Z$ has normal distribution $N(0,1)$.  Since we know  from Theorem~\ref{thm:Lepstof''} that our hypotheses provide $|\epsilon^{-2}\mathcal{L}_\epsilon- \frac16 f''(x)|\leq C\epsilon^{\alpha}$  we can conclude that
\begin{equation} \label{eqn:randomlaponinterval1}
 \sqrt{n}\epsilon^{\frac32}  \bigl( \epsilon^{-2} L_{\epsilon,n}(\omega) f(x) - \frac16 f''(x) +O(\epsilon^\alpha) \bigr) \to  \frac1{\sqrt{3}} (f'(x))(1+O(\epsilon))  Z
\end{equation}
in distribution, which is the desired result.  Since the expected size of $\bigl|\epsilon^{-2} L_{\epsilon,n}(\omega) f(x) - \frac16 f''(x)\bigr|$ has a term of size $C\epsilon^\alpha$ and another of size $C'(n\epsilon^3)^{-\frac12}$ it is minimized with  \begin{equation}\label{eq-min-eps}
\epsilon(n)=C'' n^{-\frac1{3+2\alpha}},
\end{equation}  providing the other bound in the statement.
\end{proof}

Now let us consider the difference between $\mathcal{L}_{\epsilon,n}$ and $L_{\epsilon,n}$. Both contain the sum $\sum_{E_\epsilon(x)}f(x_j(\omega))-f(x)$, which is zero if $E_\epsilon(x)=\emptyset$, but the former is normalized by the number of elements in $E_\epsilon(x)$ and the latter by dividing by $n\epsilon$. From this, for those $\omega$ such that $E_\epsilon(x)$ is non-empty, we have
\begin{align*}
\mathcal{L}_{\epsilon,n}(\omega)f(x) - L_{\epsilon,n}(\omega)f(x)
&=\Bigl( \frac1{\#E_\epsilon(x)} - \frac1{\epsilon n} \Bigr)\sum_{E_\epsilon(x)}f(x_j(\omega))-f(x)\\
&= \frac1{\epsilon} \bigl( \epsilon - \frac1n\#E_\epsilon(x)\bigr) \mathcal{L}_{\epsilon,n}(\omega)f(x)
\end{align*}
It is convenient to view $\frac1n\#E_\epsilon(x)$ as $\frac1n\sum_{j=1}^n \mathbf{1}_{B(x,\epsilon)}(x_j)$ an average of i.i.d random variables, with expectation $\mathbb{E} \mathbf{1}_{B(x,\epsilon)}(x_1)\to\epsilon$ and variance $\Var \mathbf{1}_{B(x,\epsilon)}(x_1)\to\epsilon$ as $n\to\infty$, provided that $x\in[-1+\epsilon,1-\epsilon]$. The central limit theorem then gives
\begin{equation*}
\frac{\sqrt{n}}{\sqrt{\epsilon}} \bigl( \epsilon - n^{-1}\#E_\epsilon(x)\bigr) \to Z
\end{equation*}
with convergence in distribution, where $Z$ is a standard normal, and therefore 
\begin{equation*}
	\epsilon^{-2}\bigl|\mathcal{L}_{\epsilon,n}(\omega)f(x) - L_{\epsilon,n}(\omega)f(x)\bigr|
	=\frac1{\sqrt{n\epsilon}}\frac{\sqrt{n}}{\sqrt{\epsilon}} \bigl| \epsilon - n^{-1}\#E_\epsilon(x)\bigr|
	\leq \frac{C}{\sqrt{n\epsilon}}
	\end{equation*}
as $n\to\infty$.  We conclude from~\eqref{eqn:randomlaponinterval1} that
\begin{equation*}
 \sqrt{n}\epsilon^{\frac32}  \bigl( \epsilon^{-2} \mathcal{L}_{\epsilon,n}(\omega) f(x) - \frac16 f''(x) +O((n\epsilon)^{-\frac12})+O(\epsilon^\alpha) \bigr)  \xrightarrow[\text{distr.}]{}  \frac1{\sqrt{3}} (f'(x))(1+O(\epsilon))   Z.
\end{equation*}

Since the hypotheses of Theorem~\ref{thm:randomlaponinterval} include $n\epsilon(n)^3\to0$ which implies $n\epsilon\to\infty$ and also that $[-1+\epsilon(n),1-\epsilon(n)]\to(-1,1)$  we obtain the following corollary.
\begin{cor}\label{cor:randomlaponinterval}
Under the hypotheses of Theorem~\ref{thm:randomlaponinterval}, at points $x\in(-1,1)$  we have
\begin{equation*}
\sqrt{ n\epsilon^3} \bigl( \epsilon^{-2} \mathcal{L}_{\epsilon,n}(\omega) f(x) - \frac16 f''(x) +O((n\epsilon)^{-\frac12}) + O(\epsilon^\alpha)\bigr)  \xrightarrow[\text{distr.}]{} \frac1{\sqrt{3}} |f'(x)|(1+O(\epsilon))   Z
\end{equation*}
where $Z$ is a standard normal distribution $N(0,1)$.  Minimizing on $\epsilon(n)$ we find there is a constant $C$ so
\begin{equation*}
\mathbb{E}\Bigl| \epsilon^{-2} \mathcal{L}_{\epsilon,n}(\omega) f(x) - \frac16 f''(x)\Bigr|
\leq C n^{-\frac\alpha{3+2\alpha}}.
\end{equation*}
\end{cor}

Although it is not apparent how the preceding results might be improved for low-frequency eigenfunctions, it appears from our data that in order to obtain a good approximation of the eigenmap from $[-1,1]$ to $\mathbb{R}^2$ one can take $\epsilon(n)$ much larger than is required by Theorem~\ref{thm:randomlaponinterval}.  In the next section we will see that the image under the Laplacian eigenmap is a quadratic curve (Corollary~\ref{cor:intervalunderemap}).  Some numerical data obtained by taking $n=5000$ points chosen uniformly at random in $[-1,1]$ and computing $\mathcal{L}_{\epsilon,n}$ for $\epsilon=0.01$ is in Figure~\ref{uni-fig}.  Since these are Euclidean eigenfunctions they are smooth and $\alpha=1$, so the estimate from Corollary~\ref{cor:randomlaponinterval} suggests that the error in the approximation of $f''$ could be as large as $(n\epsilon(n)^{3})^{-\frac12}$, which is approximately $14$, yet on the graphs the error in the eigenmap appears to be of the order of $0.1$, or $140$ times smaller.  We have further data that suggests the eigenmap images for Laplacians with $n$ random points and $\epsilon(n)=\frac Cn$ get closer to the expected quadratic curve as $n$ increases.  This choice of $\epsilon(n)$ is closer to that used in Proposition~\ref{prop:efnsofgraphlaplacianonI} and Corollary~\ref{cor:efnsofgraphlaplacianonI}, and leads to the question of whether the low frequency eigenfunctions of the Laplacian on an interval are already well approximated by those of random graph Laplacians $\mathcal{L}_{\epsilon(n),n}(\omega)$ with $\epsilon(n)$ comparable to $\frac Cn$, or at least a lot larger than the $n^{-\frac15}$ that would be required to apply the theorems of this section.  Motivated by this, we record the following conjecture.

\begin{figure}[h]
	\centering
	\includegraphics[width=\www]{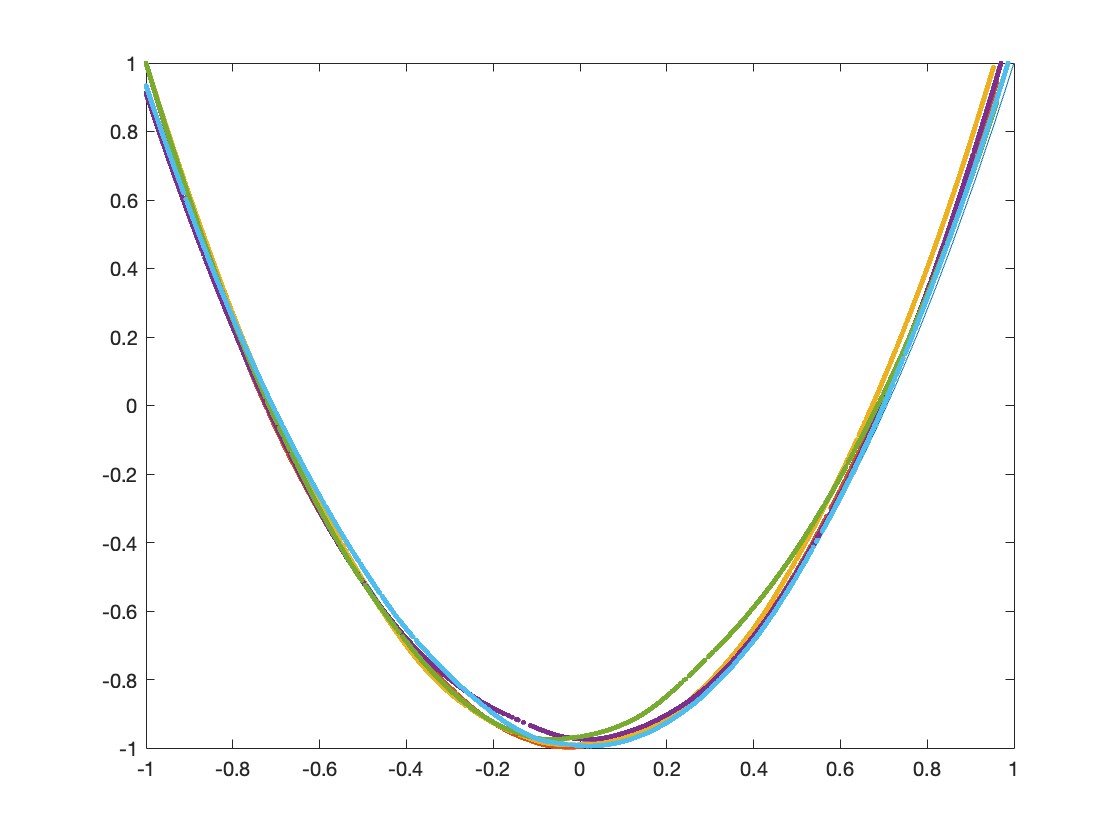}
	\includegraphics[width=\www]{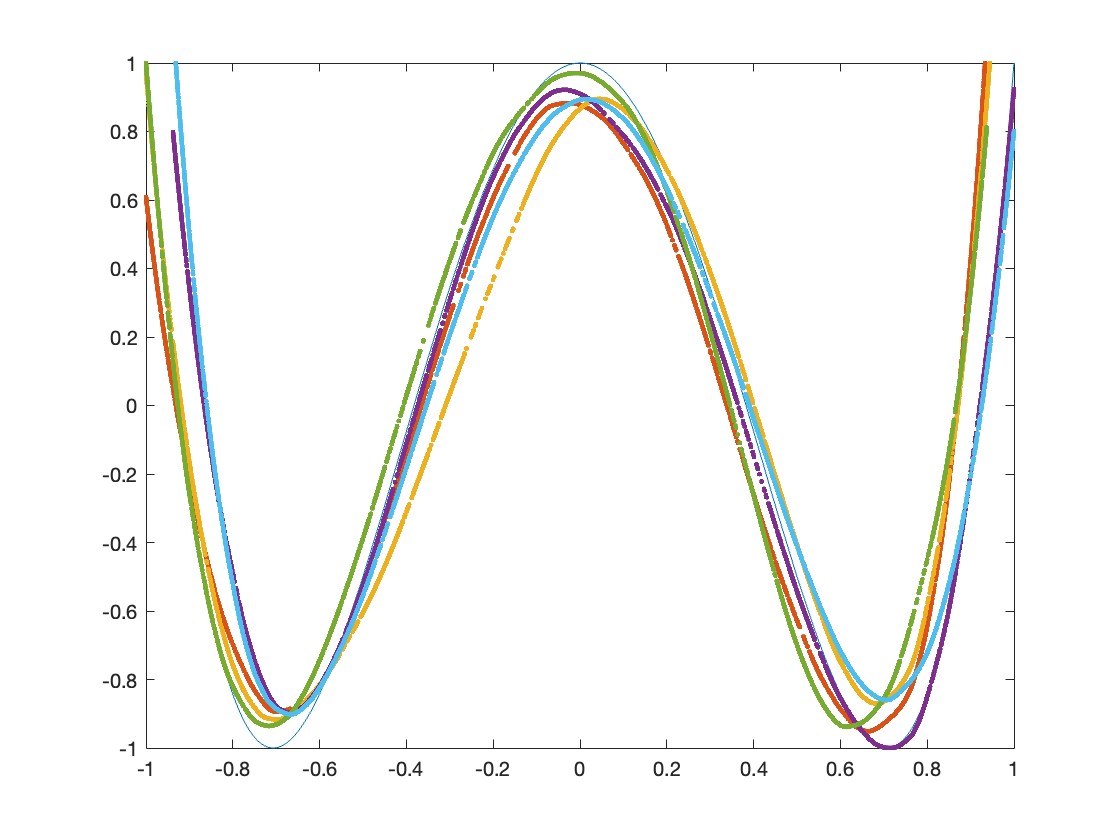}
	\includegraphics[width=\www]{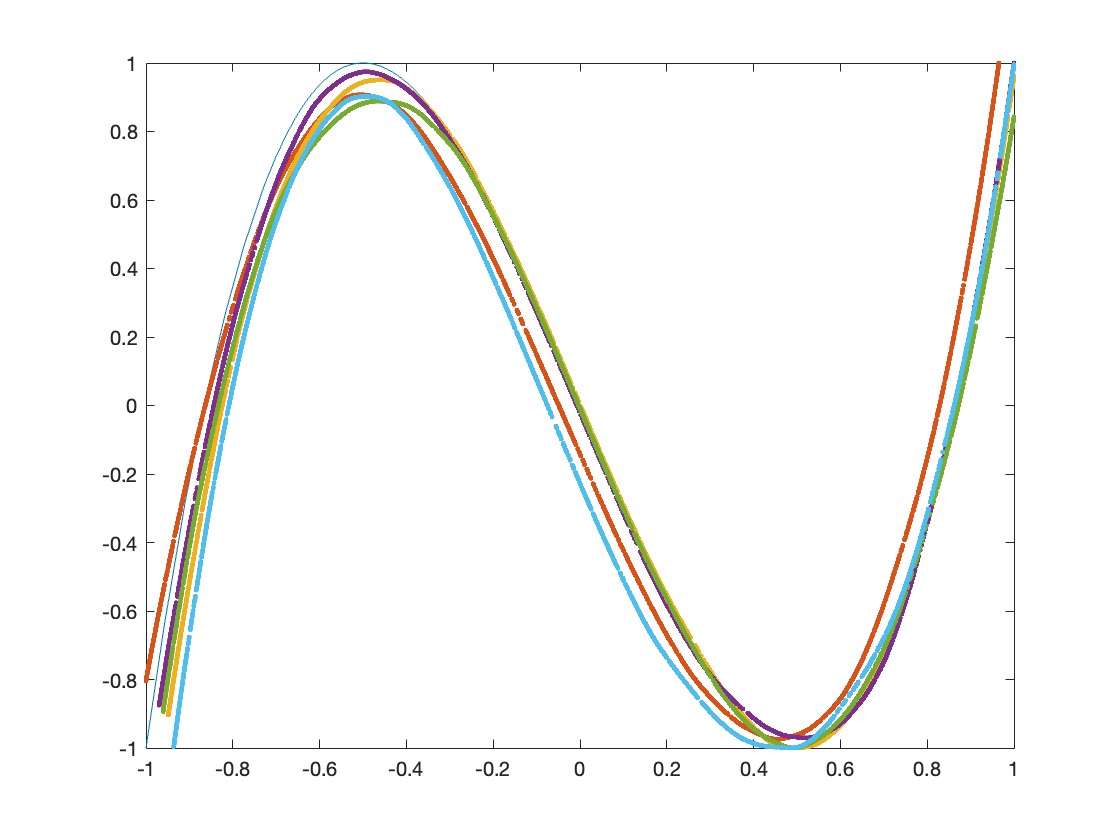}
	\caption{Laplacian eigenmaps of $n=5000,\varepsilon=0.01$, 5 runs. Top: 
		1st and 2nd eigenvectors,	$T_2(x)=2x^2-1$.
		Bottom: 	1st and 3rd eigenvectors, $T_3(x)=4x^3-3x$; 1st and 4th eigenvectors, $T_4(x)=8x^4-8x^2+1$.}
	\label{uni-fig} 
\end{figure}

\begin{conj}\label{conj:betterconvergence}
Let $\phi_k$ denote the $k^\text{th}$ eigenfunction of the Laplacian on $[-1,1]$ and $\phi_{\epsilon,n,k}$ denote the $k^\text{th}$ eigenfunction of a graph Laplacian $\mathcal{L}_{\epsilon,n}$.  If  $0<\beta < 1  $ and $\epsilon(n)\sim n^{-\beta}$ then $\|\phi_k - \phi_{\epsilon(n),n,k}\|_{L^\infty}\xrightarrow[]{}0$ as $n\rightarrow\infty$. For $0<\beta<2/3$ the convergence is improved to convergence in energy, and for $0<\beta<1/3$ the convergence is further improved to convergence of the Laplacians. 
\end{conj}

It seems possible that this conjecture might be addressed by studying the Green operator $\mathcal{L}^{-1}$ for a Laplacian with positive spectral gap, such as the Dirichlet Laplacian or the Neumann Laplacian on the quotient of its domain by constant functions.  Since the associated Dirichlet energy $\mathcal E(f):=\int_{-1}^{1} |f'|^2dx$ is a resistance form in the sense of Kigami~\cite{Kigamibook}, the Green operator can be written as an integral operator with a kernel that has a particularly simple form, see~\cite[Theorem 4.3]{Kigami2012},  \cite[and references therein]{Croydon2018}.  The same is true for $\mathcal{L}_{\epsilon,n}^{-1}$ if we permit the  connections to be between pairs of points in neighboring dyadic intervals of size comparable to $\epsilon$ instead of metric $\epsilon$-balls, so it might be possible to estimate the difference between the associated Green kernels in this case (see Section~\ref{ssec:avlaponSG} below for a description of how this would work in the case of the Sierpinski gasket, which has a cell decomposition akin to the dyadic intervals). Estimates for eigenfunctions with small eigenvalues could then be obtained from estimates for the Green operator.  This question will be considered in future work.

\subsection{Eigenmaps and orthogonal polynomials} 

As discussed in Section~\ref{ssec:IntervalLaplacians} and Proposition~\ref{prop:efnsofgraphlaplacianonI}, the eigenfunctions on $[-1,1]$ of both the Neumann Laplacian and a certain graph Laplacian with equally spaced points are the functions $\cos \frac12k\pi x$, $k$ an even integer, and $\sin \frac12 k\pi x$, $k$ an odd integer.  The image under the resulting eigenmaps then lie on the graphs of the Chebyshev polynomials.
\begin{pr}
Let $\phi_j(x)$ be the $j^{\text{th}}$ eigenfunction of the Neumann Laplacian and let $T_j$ denote the $j^{\text{th}}$ Chebyshev polynomial of the first kind.  Then $\phi_1(x)$ is the first non-constant eigenfunction and
\begin{equation*}
\phi_j(x) = \begin{cases} 
	(-1)^{\frac{j+2}2} T_j(\phi_1(x)) &\text{ if $j$ is even, and}\\
	(-1)^{\frac{j+1}2} T_j(\phi_1(x)) &\text{ if $j$ is odd.}
	\end{cases}
\end{equation*}
\end{pr}
\begin{proof}
Put $\theta=\frac12\pi(x-1)$ so that $\phi_1(x)=\sin(\theta+\frac12\pi)=\cos\theta$.  Recall that the Chebyshev polynomials of the first kind satisfy  $T_j(\cos\theta)=\cos(j\theta)$, so
\begin{align*}
\phi_{2j}(x)&=\cos(j\pi x) = \cos(2j\theta +j\pi) = (-1)^{j+1}\cos (2j\theta) = (-1)^{j+1)} T_{2j}(\cos\theta)\\
&=(-1)^{j+1}T_{2j}(\phi_1(x)),\\
\phi_{2j+1}(x)&=\sin\bigl(\frac12(2j+1)\pi x\bigr) = \sin\Bigl((2j+1)\theta+\frac12(2j+1)\pi\Bigr)\\
&= (-1)^{j+1}\cos((2j+1)\theta) = (-1)^{j+1}T_{2j+1}(\cos\theta)\\
& = (-1)^{j+1}T_{2j+1}(\phi_1(x)).\qedhere
\end{align*}
\end{proof}

Since $\phi_2(x)=T_2(\phi_1(x))=1-2(\phi_1(x))^2$ we arrive at a description of the image of our interval under the eigenmap to $\mathbb{R}^2$.

\begin{cor}\label{cor:intervalunderemap}
The image of $[-1,1]$ under the two-dimensional Laplacian eigenmap is the curve $\{(s,1-2s^2):s\in[-1,1]\}$.
\end{cor}

\begin{rem}\label{rem-ren}
A different normalization used in our numerics means that the graphs shown in Figures~\ref{fig-best-fit} and \ref{uni-fig} correspond to the image $\{(s,2s^2-1):s\in[-1,1]\}$.
\end{rem}

A similar result holds for the eigenfunctions $\psi_j$ of the Dirichlet Laplacian on $[-1,1]$, which are $\cos\frac12(k\pi x)$ for $k$ odd and $\sin\frac12(k\pi x)$ for $k$ even. In this case $\psi_1(x)=\cos\frac12\pi x $ and $\psi_{2j+1}=\cos \bigl( (2j+1)\frac12\pi x\bigr)=T_{2j+1}(\psi_1(x))$.  At the same time we have 
\begin{align*}
	\psi_{2j}(x)
	&=\sin (j\pi x)
	=\sin\bigl( ((2j-1)+1)\frac12\pi x\bigr)
	= \sin\bigl(\frac12\pi x\bigr) U_{2j-1}\Bigl(\cos\bigl( \frac12\pi x\bigr)\Bigr)\\
	&=\sqrt{1-(\psi_1(x))^2} U_{2j-1}(\psi_1(x))
	\end{align*}
where the $U_j$ are Chebyshev functions of the second kind, defined using $U_k(\cos\theta)\sin\theta=\sin((k+1)\theta)$.

Since we know that any self-adjoint Laplacian will have eigenfunctions $\phi_k(x)$ that are orthogonal, and the first non-constant eigenfunction is generally monotonic and can serve as a coordinate on $[-1,1]$, it seems plausible that the Laplacian eigenmaps on the interval arise as orthogonal polynomials for a suitably chosen measure on $[-1,1]$, however we have not succeeded in verifying this even in the case of the Laplacian $\frac{d^2}{dx^2}$ with Robin boundary conditions.

 \section{Gaussian and exponential sampling on $\mathbb{R}$}
 
  \begin{figure}[h]
  	\centering
  	\includegraphics[width=\www]{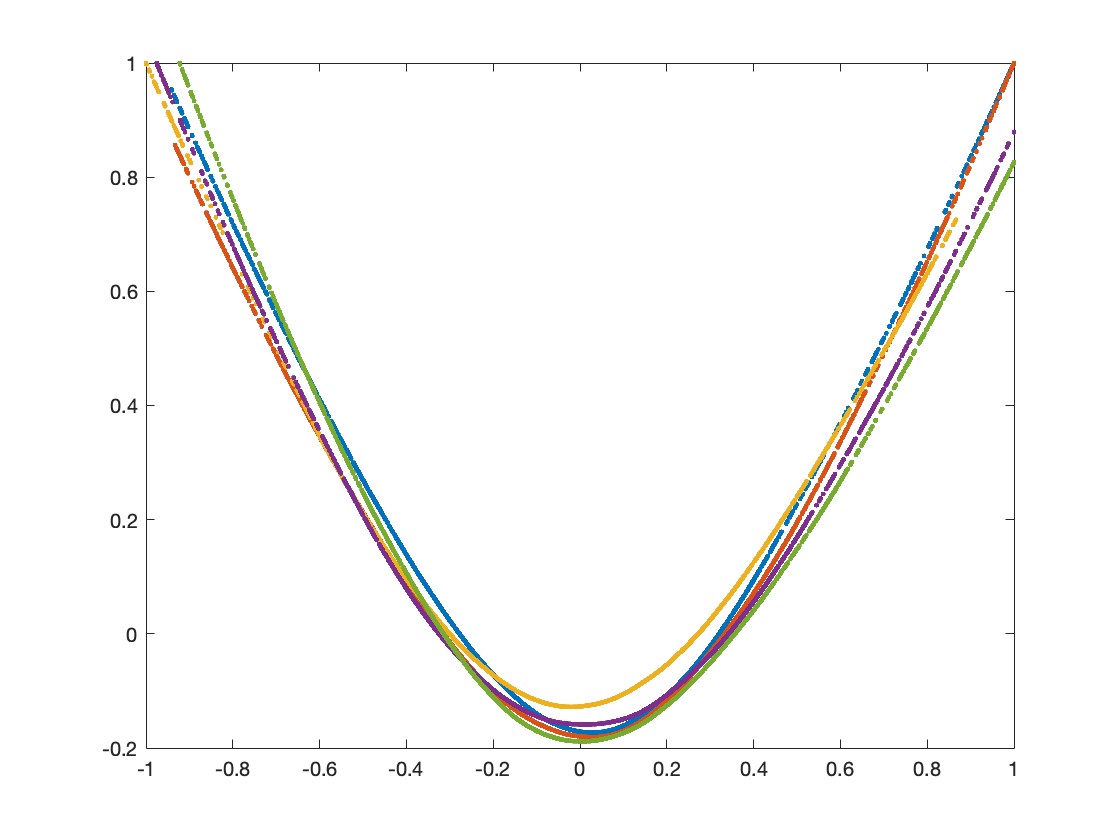}
  	\includegraphics[width=\www]{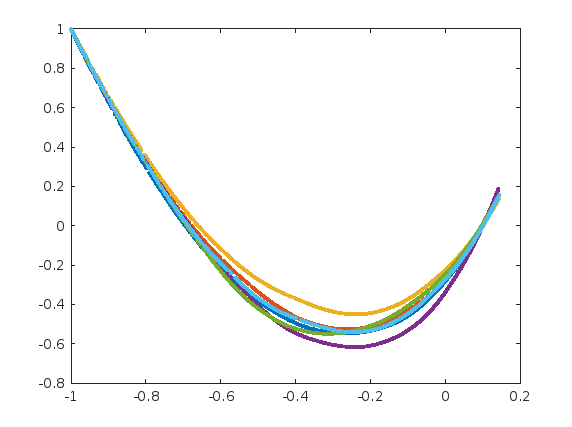}
  	\caption{Eigenmap of Gaussian (left) and exponentially (right) distributed random points}
  	\label{gauss-exp-fig}\end{figure} 
  
Figure~\ref{gauss-exp-fig} shows the image of the real line under the two dimensional eigenmap obtained from a graph Laplacian on points sampled from Gaussian and exponential distributions.  Several curves are plotted, each corresponding to a sample of points, and it appears that in each case these curves are approximating a limiting curve.  In the case of the Gaussian distribution we would expect the eigenfunctions to be Hermite polynomials and in the exponential distribution case we would expect to see Legendre polynomials, however the  numerical values in the plots differ considerably from these predictions, presumably because the the fast decay of Gaussian and exponential densities make the  kernel averaging inaccurate due to the existence of infinite regions with few or no sample points.

Our reason for expecting to see these specific polynomials is based on the following rationale.  If points are sampled from a $C^1$ density $g(x)$ then the resulting averaging Laplacian has the following form at $x=0$  if $f$ is sufficiently smooth simply by averaging the Taylor series:
 \begin{equation}\label{eq-aver-Lp}
 	\frac{\int\limits_{-\epsilon}^{\epsilon}f(x)g(x)dx }{\int\limits_{-\epsilon}^{\epsilon} g(x)dx} -f(0)
 	=\frac{\epsilon^2}6 \Bigl(  f''(0) +\frac{2g'(0)}{g(0)} f'(0) \Bigr) +o(\epsilon^2).
 \end{equation}
Normalizing by $\epsilon^{-2}$ we see this describes a second-order Laplacian with a drift term.  Using this idea, we formulate the following conjecture in the case of the real line $\R$.  For sampling on the half-line $\R_+$ (as is the case for the exponential distribution) one would need to provide suitable boundary conditions.
 
 \begin{conj}\label{conj-BN-GK}
Eigenmaps of  the random averaging Laplacians involving points sampled from a smooth distribution $g(x)$ give numerically stable locally uniform  approximations to the corresponding eigenmap of the Laplacian 
 	\begin{equation}\label{eq-Delta-g}
 		\mathcal{L}_{(g)}f:=-2\frac{g'}{  g}f'- f'' .
 	\end{equation} 
A related multidimensional case is presented in Appendix~\ref{ARTsection}. 
\end{conj}
It is worth noting that a linear change of coordinates on $\mathbb{R}$ changes $\mathcal{L}_{(g)}$ by a multiplicative factor.  More generally, on a Riemannian manifold as in Appendix~\ref{ARTsection}, the limits of averaging Laplacians are equivalent up to multiplication by a constant and rescaling of the metrics.

An integration by parts reveals that  the Laplacian in~\eqref{eq-Delta-g} is associated to a classical Dirichlet form $\mathcal{E}_g$ as discussed in~\cite{BGL,FOT} and references therein:
\begin{equation}\label{eq-weak-Lp}
 \mathcal{E}_g(f)=\int_{\R } (f'(x))^2 g(x)dx=-\int_{\R } f(x)\left( \frac{g'(x)}{g(x)}f'(x) +f''(x) \right) g(x)dx.
 \end{equation}
Since this Dirichlet form is a resistance form in the sense of Kigami~\cite{Kigamibook}, the argument suggested in the discussion following Conjecture~\ref{conj:betterconvergence} ought also to be applicable here.  In particular, this suggests approximating the Laplacian~\eqref{eq-Delta-g}by  discretization of the form~\eqref{eq-weak-Lp}.  We again formulate a conjecture in the case of $\mathbb{R}$ and note that boundary conditions would be required for an analogous conjecture on a half-line or interval.

 \begin{conj}\label{conj-BGLK}
Numerically accurate  locally uniform  approximations to the Laplacian eigenmap can be obtained by  discretization of the left and right hand sides in the weak  definition of the Dirichlet form Laplacian \eqref{eq-weak-Lp} using the theory of resistance forms~\cite{Kigamibook}. 
 \end{conj}
 
As noted at the beginning of this section, we believe that in the case of Gaussian and exponential distributions, we believe that Conjectures\ref{conj-BN-GK} and ~\ref{conj-BGLK} should lead to approximations of the classical Hermite and Legendre polynomials.

 \section{Eigenmaps on fractal spaces}\label{fractalsection}
 
\subsection{Convergence of graph Laplacians with well-spaced points to averaging Laplacians on self-similar sets}\label{ssec:fractalswithwellspacedpts}

Suppose $(Y,d)$ is a complete metric space and let $X$ be the unique compact invariant self-similar set associated to a contractive iterated function system $\{F_1,\dotsc, F_N\}$ of strictly contractive homeomorphisms that satisfy the open set condition~\cite{Hutchinson81}.  Writing $F_w=F_{w_1}\circ\dotsm F_{w_n}$ for a word $w_1\dotsm w_m$ of length $|w|=m$ with letters from $\{1,\dotsc, N\}$ we have $X=\cup_{m\geq0}\cup_{\{|w|=m\}} F_w(X)$.  Evidently  
\begin{equation*}
\max_{|w|=m}\diam(F_w(X))\to0 \text{ as } m\to\infty.
\end{equation*}
 The open set condition ensures we may endow $X$ with the uniform self-similar probability measure in which $\mu(F_w(X))=N^{-m}$ and $\mu(F_w(X)\cap F_{w'}(X))=0$ if $|w|=|w'|$ and $w\neq w'$, see~\cite{Hutchinson81}.  Evidently this construction provides $(X,d,\mu)$ with measure-theoretic equipartitions of arbitrarily small diameter as defined immediately preceding Corollary~\ref{cor:graphlapapproxavlapifmeasurethpartitions}, because cells $F_w(X)$ with $|w|=m$ form a measure-theoretic equipartition and have diameter converging to zero with $m$.  Applying  Corollary~\ref{cor:graphlapapproxavlapifmeasurethpartitions} we obtain the following.

\begin{pr}\label{prop:selfsimilarwellspacedpts}
If $(X,d,\mu)$ is a self-similar set with diameter at least $1$ and uniform probability measure as described above, then for $\epsilon>0$ and $n\in\mathbb{N}$ there is an averaging Laplacian $\mathcal{L}_\epsilon$ over cells in the sense of Definition~\ref{def:avlap} and graph Laplacian operators $\mathcal{L}_{\epsilon,n}$ satisfying
\begin{equation*}
\bigl|\mathcal{L}_\epsilon f(x) - \mathcal{L}_{\epsilon,n}f(x) \bigr|
\leq \sup_x \sup_{y\in B(x,\frac\epsilon n)} |f(x)-f(y)|.
\end{equation*}
\end{pr}

Typical examples of sets to which the preceding applies include the classical Sierpinski gasket and carpet, Menger sponges, and the Koch curve.  In particular, we find that taking $n=N^m$ graph vertices with one vertex in each $m$-scale cell will provide $\mathcal{L}_{\epsilon,n}f\to\mathcal{L}_\epsilon f$ pointwise.

\subsection{Random sampling from a fractal, $\mathcal L_{\epsilon,n}\xrightarrow[n\to\infty]{}\mathcal L_{\epsilon}$ }  

With $(X,d,\mu)$ as in the previous section, take points $\{x_1,\dotsc,x_n\}$ independently sampled from $X$ according to $\mu$. We define a random graph Laplacian in a similar manner to Definition~\ref{def:graphlaplacians} but considering cells of scale comparable to $\epsilon$ rather than balls.  For this purpose it is convenient to assume that all cells of scale $m$ have the same diameter $r^m$, where $r\in(0,1)$ is fixed, though more general circumstances can be analyzed.  For $\epsilon>0$ take $m$ so $r^{m}\leq\epsilon\leq r^{m-1}$ and let $E_m(x)$ be the $m$-cell containing $x$, then define the random graph Laplacian
\begin{equation*}
     \mathcal{L}_{\epsilon,n}(\omega) f(x)= \frac{1}{\#\{x_j\in E_m(x)\}} \sum_{x_j\in E_m(x)} (f(x_j)-f(x)).
\end{equation*}
where, as usual, $\omega\in\Omega$ describes the random choice of sample points.  Since $\frac1n\#\{x_j\in E_m(x)\}\to \mu(E_m(x))=N^{-m}$ this is close to  $N^{-m}L_{\epsilon,n}$ where 
\begin{equation*}
     L_n(\omega)f(x)= \frac{1}{n}\sum_{x_j\in E_m(x)} (f(x_j)-f(x))
\end{equation*}
When analyzing the corresponding object on an interval in $\mathbb{R}$ in Section~\ref{ssec:iidlaponinterval} we had access to the central limit theorem to determine the rate of convergence of $L_n(\omega)f(x)$ to the averaging Laplacian $\mathcal{L}_\epsilon f$ over $m$ cells.  For a general fractal we do not have such a result, but it is still possible to determine that the (random) eigenvalues of $L_n(\omega)f$ converge to the (non-random) eigenvalues of $\mathcal{L}_\epsilon f$ using results from~\cite{koltchinskii2000random}

More precisely, the operator $\mathcal{L}_\epsilon+I$ simply averages a function over $n$-cells. Since all cells have the same measure this can be  viewed as an integral operator $f\mapsto \int h(x,y)f(y)d\mu$ where $h(x,y)=N^{-m}\mathbf{1}_{E_m(x)}(y)$ is a symmetric kernel.  This is a Hilbert-Schmidt operator, so by~\cite[Theorem~3.1]{koltchinskii2000random} we have the following.

 \begin{thm}
For fixed $\epsilon>0$ and $n\to\infty$ the random eigenvalues  of $L_{\epsilon,n}(\omega)$ converge to the non-random eigenvalues of $\mathcal L_{\epsilon}$ almost surely as $n\to\infty$.
\end{thm}
In the future work it would be interesting to investigate the rate of convergence of eigenvalues and to use the techniques of~\cite{Koltchinskii,KoltchinskiiLounici} to explore the convergence of eigenfunctions and, more generally, eigenprojectors.  The approach of~\cite{burago2019spectral} might also be applicable in this setting.

\subsection{Convergence of averaging Laplacians $\mathcal L_{\epsilon}\xrightarrow[\epsilon\to0]{}\mathcal L$ on the Sierpinski gasket} \label{ssec:avlaponSG}
 
For our purposes we take the  Sierpinski gasket SG to be the unique non-empty compact set in $\mathbb{R}^2$ which is fixed by the contractions $F_j=\frac12(x-p_j)+p_j$, $j=1,2,3$ with $p_1=(0,0)$, $p_2=(1,0)$, $p_3=\bigl(\frac12,\frac{\sqrt{3}}2\bigr)$ the vertices of an equilateral triangle. This is a special case of the situation described in Section~\ref{ssec:fractalswithwellspacedpts} and we use the notation established there.
 We endow SG with the self-similar measure $\mu$ that satisfies $\mu=\frac13\sum_1^3 \mu\circ F_j^{-1}$.

The Sierpinski gasket can be provided with a natural Laplacian $\mathcal{L}$ (and an associated heat semigroup and Dirichlet form) using either probabilistic or functional analytic methods~\cite{BarlowPerkins,Kigamibook}.  This Laplacian is non-positive definite and has discrete spectrum; moreover, there is a rapid algorithm for computing eigenfunctions, so it is not difficult to determine the Laplacian eigenmap numerically to any desired accuracy, and one finds that the image of the gasket under the eigenmap to $\mathbb{R}^2$ is recognizably a gasket.  However we cannot use any results regarding the Laplacian on a manifold to understand this eigenmap because $\mathcal{L}$ is not a Laplace-Beltrami operator. 

The analysis associated to the Laplacian $\mathcal{L}$ provides for a type of Taylor approximation of a function $f$ for which $\mathcal{L}f$ is continuous, but we do not expect that the averaging Laplacian with respect to balls in the metric obtained from $\mathbb{R}^2$ can be rescaled to converge to $\mathcal{L}$ in the strong sense that occurs for the Laplace-Beltrami operator on a manifold.  A more detailed description of what we expect is in Conjecture~\ref{conj-no-conv3}.  The reason is that Euclidean balls are not the correct sets for the cancellation of the lower order terms to occur.

\begin{figure}[h] 
	\centering
	\includegraphics[width=.55\textwidth]{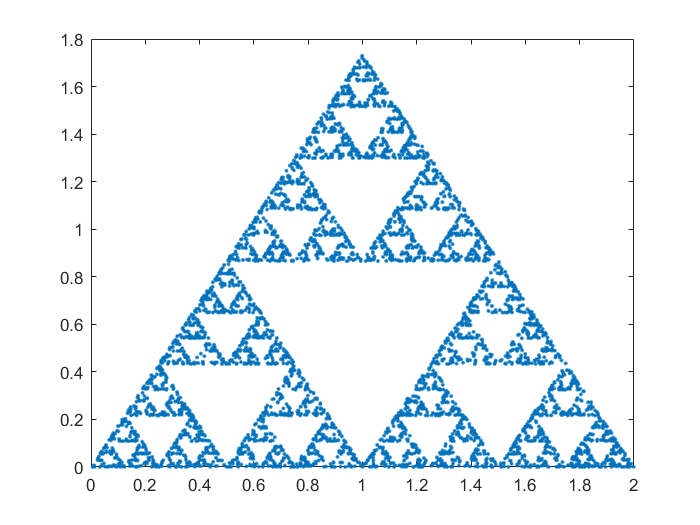}
	\includegraphics[width=.55\textwidth]{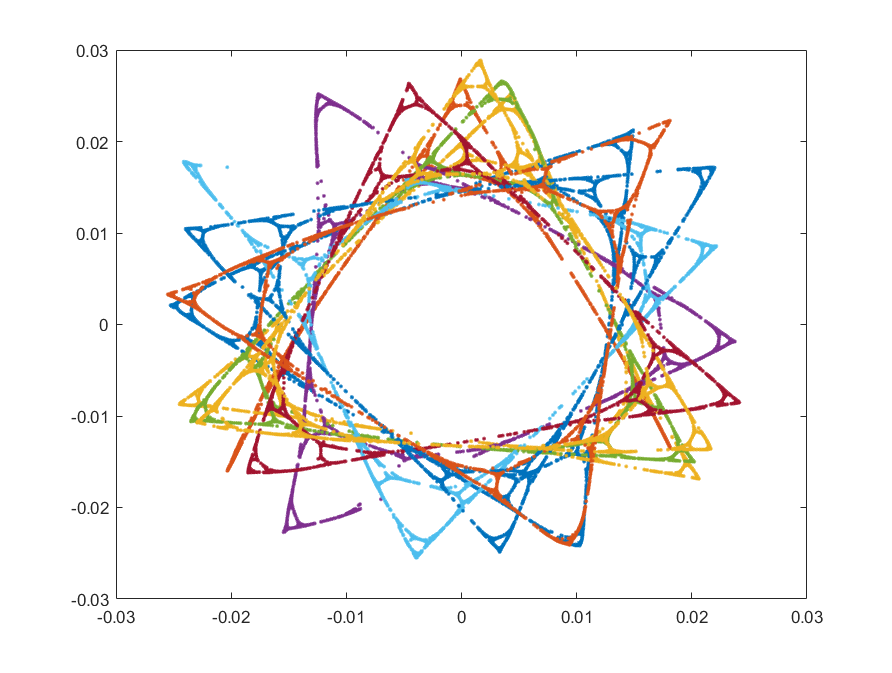}
	\includegraphics[width=.55\textwidth]{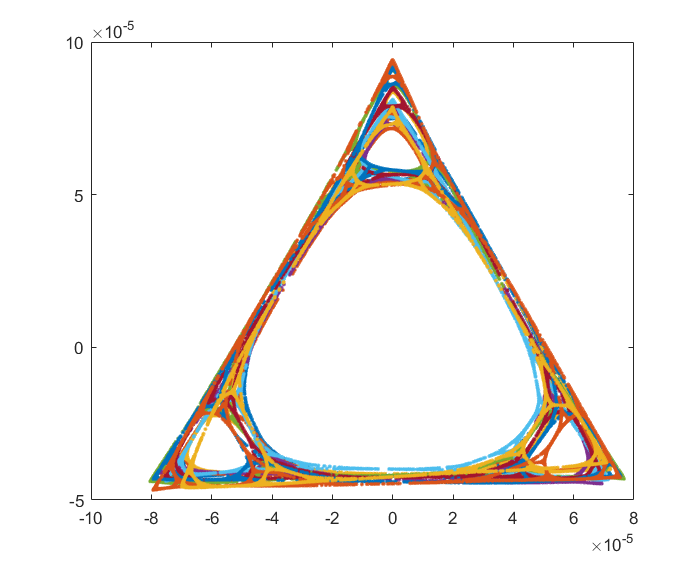}
	\caption{
		Top: the Sierpinski Gasket with 5000 points and  $m = 15$.  Middle: Sierpinski Eigenmaps plotted over each other without any extra rotation, supporting Conjecture~\ref{conj-conv-rot}. 
		Bottom: 
		Sierpinski Eigenmaps plotted over each other each, rotated through an extra linear transformation to align, supporting Conjectures~\ref{Sierpinski-conj} and \ref{conj-conv2}.  }  \label{fractal-fig}
\end{figure} 

It is known, however, that there is a way to obtain $\mathcal{L}$ using averaging.  One can define an averaging Laplacian with respect to cells $F_w(\text{SG})$, and this can be rescaled to converge to $\mathcal{L}$, see~\cite{strichartz2001laplacian,donglei2005laplacian,donglei2007laplacian}.  It is important to note the the rescaling is different than occurs in the Euclidean space: if $|w|=m$ the diameter of $F_w(\text{SG})$ in $\mathbb{R}^2$ is $2^{-m}$ so in the manifold case we would rescale the averaging Laplacian by the squared reciprocal $4^m$ to obtain the Laplacian in the limit as $m\to\infty$, but in the case of SG we rescale instead by $5^m$.  We do not explain this phenomonen here, but simply note that it indicates that $\mathcal{L}$ is distinctly different from a Euclidean Laplacian and one should not expect to transfer arguments from Euclidean space or Riemannian manifolds directly to this setting.

Since the averaging Laplacian over cells can be rescaled to converge to $\mathcal{L}$ it seems possible that one would still see this convergence for an averaging Laplacian $\mathcal{L}_\epsilon$ defined following Definition~\ref{def:avlap} but using a notion of distance that reflects the cell structure.   For the latter, we consider the following.

\begin{df}
 	The cellular semimetric on the Sierpinski gasket is defined by 
\begin{equation*}
d_{\text{cell}}(x,y):=\min\{2^{-m}: x \text{ and } y \text{ are in intersecting $m$-cells}\}.
\end{equation*}
 \end{df}
To see that the triangle inequality fails, making this a semi-metric, consider the distance between $p_1$ and two points $x,y$ with $x\in F_1(X)$, $y\in F_2(X)$ so $d_{\text{cell}}(p_1,x)=\frac12$, $d_{\text{cell}}(p_1,y)=1$.  By making $x$ and $y$ lie in $k$-cells that meet at $F_1(p_2)$ we can ensure $d_{\text{cell}}(x,y)=2^{-k}$ for arbitrarily large $k$, so $d_{\text{cell}}(p_1,y)\not\leq d_{\text{cell}}(p_1,x)+d_{\text{cell}}(x,y)$.  It is not hard to see that the Euclidean distance and cellular semimetric are related by $\frac12 d_{\text{cell}}(x,y)\leq|x-y|\leq 2d_{\text{cell}}(x,y)$.  We note that $d_{\text{cell}}$ is closely related to $\tilde{d}(x,y)=\min\{2^{-m}:\text{$x$ and $y$ are in the same $m$-cell\}}$ and, when restricted to intersection points of cells, to the Hamming distance.


We define an  averaging Laplacian in the sense of Definition~\ref{def:avlap} by
\begin{equation*}
\mathcal{L}_\epsilon f(x)=\frac1{\mu(\{y:d_{\text{cell}}(x,y)\leq\epsilon\})} \int_{\{y:d_{\text{cell}}(x,y)\leq\epsilon\}} \bigl( f(y)-f(x)\bigr)d\mu. 
\end{equation*}

The existing results on obtaining the Laplacian using rescaled cell averages suggest the following should be true.
\begin{conj}\label{Sierpinski-conj}
There is a constant $c$ so that if the averaging operator $\mathcal L_{\epsilon}$ on  SG is defined with respect to $d_{\text{cell}}$  then 
\begin{equation*}
\epsilon^{-\frac{\log 5}{\log2}} \mathcal L_{\epsilon}f\xrightarrow[\epsilon\to0]{} c\mathcal Lf
\end{equation*}
uniformly for any function $f$ in the H\"older domain of the Laplacian. Moreover, the eigenvalues and eigenfunctions of $\mathcal{L}_\epsilon$ converge to those of $\mathcal{L}$ as $\epsilon\to0$.  	
\end{conj}

We believe it may be possible to prove this Conjecture~\ref{Sierpinski-conj} either by using the Dirichlet form methods discussed after Conjecture~\ref{conj:betterconvergence} or using results on obtaining the Laplacian by the method of averages~\cite{strichartz2001laplacian,donglei2005laplacian,donglei2007laplacian}.  The latter might also offer a way to investigate the rate of convergence.

Next we consider  points $\{x_1,\dotsc,x_n\}$ randomly sampled from SG according to $\mu$ and define a random graph Laplacian using the semimetric $d_{\text{cell}}$ by
\begin{equation*}
\mathcal{L}_{\epsilon,n}(\omega) f(x)
= \frac1{\#\{x_j:d_{\text{cell}}(x,x_j)\}} \sum_{\{x_j:d_{\text{cell}}(x,x_j)\}} f(x_j)-f(x_k).
\end{equation*}
In light of the analysis leading to Proposition~\ref{prop:selfsimilarwellspacedpts} we see that $\mathcal{L}_{\epsilon,n}$ should, in a suitable average sense, converge to $\mathcal{L}_\epsilon$ and therefore, after rescaling, to $\mathcal{L}$.  By analogy with Corollary~\ref{cor:randomlaponinterval} we expect a result as follows. Note that the normal, or Neumann, derivative (for which see~\cite[Definition~3.7.6]{Kigamibook}) at a point $p_j$ plays an analogous role in the theory to $f'(\pm1)$ does in  Theorem~\ref{thm:randomlaponinterval}.

\begin{conj}\label{conj:SGlapbyrandgraphs}
Suppose $f:\text{SG}\to\mathbb{R}$ has vanishing normal derivatives at $\{p_1,p_2,p_3\}$ and $\mathcal{L}f$ is H\"older continuous with exponent $\alpha$.  There is $\beta>1$ and $C>0$ so that if $\epsilon(n)\to0$ satisfies $n\epsilon(n)^\beta\to\infty$ then at each $x\in\text{SG}\setminus\{p_1,p_2,p_3\}$ we have
\begin{equation*}
\mathbb{E}\Bigl| \epsilon^{-\frac{\log 5}{\log2}} \mathcal L_{\epsilon}(\omega) f -  c\mathcal Lf\Bigr|
\leq C n^{-\gamma}
\end{equation*}
where $\gamma=\gamma(\alpha,\beta)>0$.
\end{conj}

The fact that the first Neumann eigenspace for $\mathcal{L}$ is two dimensional should imply that the graph Laplacians at randomly sampled points $\mathcal{L}_{\epsilon,n}(\omega)$ will have two eigenspaces with eigenvalues that are close together; the corresponding eigenfunctions are orthogonal and should both be close to eigenfunctions from the first Neumann eigenspace of $\mathcal{L}$.  In the limit as $n\to\infty$ they should then look like a random choice of basis for $\mathbb{R}^2$, providing the following conjecture that is also supported by numerical evidence for the random graph Laplacians defined using Euclidean distance, see the central image in Figure~\ref{fractal-fig}.

 \begin{conj}\label{conj-conv-rot}
 	As $n\to\infty$ and $\epsilon(n)\to0$ suffcently slowly, 
 	eigenmaps for random graph approximations to $\mathcal{L}$ on the Sierpinski gasket produce images of nonrandom eigenmaps that are randomly rotated  in $\mathbb{R}^2$ due to the relevant Neumann eigenvalues for $\mathcal{L}$ having multiplicity two. We conjecture that these random rotations have a limiting density which is invariant under the group of symmetries of the Sierpinski gasket. 
 \end{conj}
 
\begin{conj}\label{conj-conv2}
	For both the rescaled averaging Laplacian and the rescaled graph Laplacian at random sample points, defined with respect to the Euclidean metric on SG,  the eigenvalues and eigenfunctions converge uniformly to those of $\mathcal{L}$ with proper choice of normalization and alignment. Here $n\to\infty$ and $\epsilon(n)\to0$ but we do not know optimal   behavior of $\epsilon(n)$ for most accurate convergence. 
\end{conj}

The preceding conjectures are mostly motivated by the existing results of~\cite{strichartz2001laplacian,donglei2005laplacian,donglei2007laplacian}, the law of large numbers and Proposition~\ref{prop:selfsimilarwellspacedpts}.  Our other conjectures arise from our numerical investigations of the eigenmaps of graph Laplacians at random sample points as summarized in Figure~\ref{fractal-fig}, which refer to graph Laplacians $\mathcal{L}_{\epsilon,n}$ defined using the Euclidean distance rather than the semimetric $d_{\text{cell}}{x,y}$. Given the Lipschitz bound $\frac12 d_{\text{cell}}(x,y)\leq |x-y|\leq 2d_{\text{cell}}(x,y)$ these are expected to be very similar. Indeed we expect that Conjecture~\ref{conj-conv-rot} will hold for these random graph Laplacians, as is suggested by the middle image in Figure~\ref{fractal-fig} and, in light of the last image in Figure~\ref{fractal-fig}, that there will also be convergence of eigenvalues and eigenfunctions.
 
To conclude this section we specify our prediction for the manner in which averaging over Euclidean balls should fail to approximate the Laplacian in the case of SG.

\begin{conj}\label{conj-no-conv3}
Suppose $\mathcal{L}f$ is H\"older continuous on SG.  If $\mathcal{L}_\epsilon$ and $\mathcal{L}_{\epsilon,n}(\omega)$ are defined using the Euclidean metric then there is a sequence $\epsilon_j\to0$ such that $\epsilon_j^{-\frac{\log 5}{\log2}} \mathcal L_{\epsilon_j}(\omega) f \to  c\mathcal Lf$ as $j\to\infty$ in the weak-$L^2$ topology (and similarly for $\mathbb{E}\mathcal{L}_{\epsilon,n}(\omega)$), but this convergence does not hold in the strong-$L^2$ topology for any sequence $\epsilon_j$.
\end{conj}

See \cite{Kumagai1996,Kumagai1998,Kumagai2004} for related work on homogenization on fractals. 

\section{Weighted Riemannian manifolds with boundary}\label{ARTsection}

 In forthcoming article \cite{ART} we extend results of \cite{gine2006empirical} to an wide class of Riemannian manifolds with boundary. Below we give state two of the main theorems using notation from \cite{gine2006empirical}: 
 \begin{equation}
 	\label{eqn1}
 	D_{\epsilon,n}f(p):=\frac{1}{n\epsilon^{d+2}}\sum_{j=1}^nK(\frac{p-X_j}{\epsilon})(f(X_j)-f(p)),
 \end{equation} 
 \begin{equation}
 	\label{eqn2}
 	D_\epsilon f(p) :=\frac{1}{\epsilon^{d+2}}\mathbb{E}K(\frac{p-X}{\epsilon})(f(X)-f(p)) 
 \end{equation}
where $X,X_j,p\in\mathbb R^d$. 

As in the previous sections, we consider the situation where \(\epsilon = \epsilon(n) \to 0\) as \(n \to \infty\). Before discussing general Riemannian manifolds with boundaries, we present the most instructive one-dimensional case. For comparison with \cite{gine2006empirical}, we do not assume that functions \(f\) have bounded third-order derivatives, nor do we assume that \(n\epsilon^{1+4} \to 0\). Thus, our result in Theorem~\ref{thm1} strengthens the corresponding result in \cite{gine2006empirical}.

\begin{thm}[{\cite[Theorem 2.1]{ART}}]
\label{thm1}
Let $d=1$, $\{X,X_j\}$ be i.i.d with bounded   continuous density g.
Suppose that $$n\epsilon \to \infty$$ and $K\in L^1(\mathbb{R})$  
satisfies  
\begin{equation}\label{eq-Kt4}
 \int_{\mathbb R} (t^4+t^2) (K^4(t)+K^2(t))dt<\infty.  
\end{equation}

If   $f \in W^{2,\infty} (\mathbb R)$  
then 
\begin{equation} \label{eq-Zn}
\sqrt{n\epsilon^{1+2}}(D_{\epsilon,n}f(p)-D_{\epsilon}f(p)) \xrightarrow[n\to\infty]{d} s\mathcal{Z}
\end{equation}
where  $\mathcal{Z} \sim N(0,1)$ and 
\[
s^2=(f'(p))^2g(p)  \int_\mathbb{R}K^2(t)t^2 d\lambda(t).
\]\end{thm}

\begin{proof}
In this proof, without loss of generality, we assume that $p=0$ and $f(0)=0$. 
Let $Z_n$ denote the left hand side of \eqref{eq-Zn}. 

We start with considering the case $f'(0)=0$. In that case 
$
|f(x)| \leqslant    \|f''  \|_\infty |x|^2
$
and so 
\begin{align}
\mathbb{E}(Z_{n}^2)&=\frac{1}{n\epsilon^{1+2}}\sum_{j=1}^n\mathbb{E}\Big(D_{\epsilon,n}f(0)-\mathbb E 
\left(D_{\epsilon}f(0)\right)
)^2\\
&\leqslant
\frac{1}{ \epsilon^{1+2}} \mathbb{E}\Big(D_{\epsilon,n}f(0) \Big)
^2 
\leqslant 
\frac{\|f''  \|_\infty^2}{ \epsilon^{1+2}} \mathbb{E} K^2(  X /\epsilon)X^4 
  \\
&
\leqslant
 \frac{C^2}{ \epsilon^{1+2}} \int _{\mathbb R}  K^2(  x /\epsilon)x^4  dx
=
\epsilon^2 C^2   \int _{\mathbb R}  K^2(  t)t^4  dt \xrightarrow[\epsilon\to0] {}0
\end{align}
  because of \eqref{eq-Kt4}.

If $f'(0)\neq0$  then, without loss of generality, in the end of the proof it is enough to consider the case $f(x)=x$ and $K$ not equal to zero. In that case 
\begin{align}
Z_{n}&= \frac{1}{\sqrt{n\epsilon^{1+2}}}\sum_{j=1}^n
\left(K
\left(\frac{-X_j}{\epsilon}\right)
X_j-
\mathbb{E}
\left(K
\left(\frac{-X_j}{\epsilon}\right)X_j\right) \right) 
\\
Var(Z_{n})& = \frac{1}{\epsilon^{1+2}} 
\mathbb{E}
\left(K
\left(\frac{X}{\epsilon}\right)
X-
\mathbb{E}
\left(K
\left(\frac{X}{\epsilon}\right)X\right) \right)  ^2
\\
&
\xrightarrow[\epsilon\to0] {}
 g(0) \int_{\mathbb{R}} t^2  K^2(t)  \, dt  .
\end{align}
This quantity is positive unless $K$ vanishes identically or $g(0)=0$, and here we used the fact that 
\begin{equation}
\lim\limits_{ \epsilon\to0 }
\frac{1}{\sqrt{\epsilon^{1+2}}} 
\mathbb{E}
\left(K
\left(\frac{X}{\epsilon}\right)X\right)
= const
\lim\limits_{ \epsilon\to0 }\sqrt{\epsilon} = 0.
\end{equation}
To use Lyapounov's $CLT$ from \cite{billingsley2017probability}, it suffices to check 
\begin{equation}
\lim_{n\to\infty} 
\frac{1}{n\epsilon^{2(1+2)}}\mathbb{E}
\left(K
\left(\frac{X}{\epsilon}\right)
X-
\mathbb{E}
\left(K
\left(\frac{X}{\epsilon}\right)X\right) \right)   
^4
=
\lim_{n\to\infty} 
\frac{const}{n\epsilon}
=
0
\end{equation} 
which follows from calculations   above and the condition   $n\epsilon\to\infty$.
\end{proof}

\begin{rem}
Theorem~\ref{thm1} allows for the kernel $K(t, p)$ to depend on the point $p$, in addition to its dependence on $t$. We choose to omit this dependence in our notation because it does not influence the proof.
\end{rem}

 In higher dimensions we have the following theorems. 
 
 \begin{thm}[{\cite[Theorem 2.4]{ART}}]
 	\label{thm2.4}
 	Let $\{X,X_j\}$ be i.i.d random variables taking values in $\mathbb{R}^d$ with continuous   bounded density $g$. 
 	Suppose $K:\mathbb{R}^d\to\mathbb{R}$ is a locally bounded kernel satisfying $O(\frac{1}{\norm{t}^{d+4}})$ as $\norm{t} \to \infty$.
 	Let $\mathcal{F}$ be a family of real valued functions on $\mathbb{R}^d$ 
 	with bounded derivatives of the third order, $f:\mathbb{R}^d \to \mathbb{R}$ is such that $f\in \mathcal{F}$, and $D_{\epsilon,n}f(p)$, and $D_{\epsilon}f(p)$ be as in \ref{eqn1} and \ref{eqn2} respectively.
 	If $n\epsilon^d \to \infty$ and $n\epsilon^{d+4} \to 0$ then 
\begin{equation}
 	\sqrt{n\epsilon^{d+2}}(D_{\epsilon,n}f(p)-D_{\epsilon}f(p)) \overset{d}\to s\mathcal{Z}
\end{equation} 	
 	where, $\epsilon$ is a sequence such that $\epsilon \overset{n\to \infty} \to 0$, $\mathcal{Z} \sim N(0,1)$ and 
\begin{equation}
 	s^2 = g(p) \sum_{i,j=1}^d  \frac{\partial}{\partial x_i}f(p)\frac{\partial}{\partial x_j}f(p)\int_{\mathbb{R}^d}K^2(-t)t_it_j d\lambda(t).
\end{equation} 	 \end{thm}

 \begin{thm}[{\cite[Theorem 2.5]{ART}}]
 	\label{thm0.4}
 	Let $X$ be a random variable  taking values in $\mathbb{R}^d$ with twice differentiable density $g$, bounded up to the second order.
 	Suppose $K$ is an even, locally bounded kernel satisfying $K(t)=O(\frac{1}{\norm{t}^{d+5}})$.
 	If $\mathcal{F}$ is a family of three times differentiable functions with bounded derivatives of the third order, $f \in \mathcal{F},D_\epsilon f(p)$ is defined as before and 
 	\begin{align} \label{delta}
 		\Delta f(p)&:=\sum_{i,j=1}^d\frac{\partial}{\partial x_i}f(p)\frac{\partial}{\partial x_j}g(p)\int_{\mathbb{R}^d}K(-t)t_it_jd\lambda(t) \\
 		&\quad + \frac{g(p)}{2}\sum_{i,j=1}^d\frac{\partial^2}{\partial x_i \partial x_j}f(p)\int_{\mathbb{R}^d}K(-t)t_it_jd\lambda(t)
 	\end{align}
 	Then
 	\begin{equation*}
 		\abs{D_\epsilon f(p)-\Delta f(p)}=O(\epsilon).
 	\end{equation*}
 \end{thm} 
  
For the sake of brevity, we do not include here the results for manifolds with boundary, which are detailed in \cite{ART}.

\bibliographystyle{alpha}    
\bibliography{REU2023bib}

\newcommand{\etalchar}[1]{$^{#1}$}
\begin{thebibliography}{GTGHS20}

\bibitem[ART24]{ART}
Bernard Akwei, Luke Rogers, and Alexander Teplyaev.
\newblock Distributional convergence of the empirical laplacian with decaying
  kernels on domains with boundaries.
\newblock {\em preprint}, 2024.

\bibitem[BGL{\etalchar{+}}14]{BGL}
Dominique Bakry, Ivan Gentil, Michel Ledoux, et~al.
\newblock {\em Analysis and geometry of {M}arkov diffusion operators}, volume
  103.
\newblock Springer, 2014.

\bibitem[BIK15]{burago2015graph}
Dmitri Burago, Sergei Ivanov, and Yaroslav Kurylev.
\newblock A graph discretization of the {L}aplace--{B}eltrami operator.
\newblock {\em Journal of Spectral Theory}, 4(4):675--714, 2015.

\bibitem[BIK19]{burago2019spectral}
Dmitri Burago, Sergei Ivanov, and Yaroslav Kurylev.
\newblock Spectral stability of metric-measure laplacians.
\newblock {\em Israel Journal of Mathematics}, 232(1):125--158, 2019.

\bibitem[Bil17]{billingsley2017probability}
Patrick Billingsley.
\newblock {\em Probability and measure}.
\newblock John Wiley \& Sons, 2017.

\bibitem[BN03]{belkin2003laplacian}
Mikhail Belkin and Partha Niyogi.
\newblock Laplacian eigenmaps for dimensionality reduction and data
  representation.
\newblock {\em Neural computation}, 15(6):1373--1396, 2003.

\bibitem[BN08]{belkin2008towards}
Mikhail Belkin and Partha Niyogi.
\newblock Towards a theoretical foundation for {L}aplacian-based manifold
  methods.
\newblock {\em Journal of Computer and System Sciences}, 74(8):1289--1308,
  2008.

\bibitem[BP88]{BarlowPerkins}
Martin~T. Barlow and Edwin~A. Perkins.
\newblock Brownian motion on the {S}ierpi\'{n}ski gasket.
\newblock {\em Probab. Theory Related Fields}, 79(4):543--623, 1988.

\bibitem[Chr90]{christ1990b}
Michael Christ.
\newblock A {$T(b)$} theorem with remarks on analytic capacity and the {C}auchy
  integral.
\newblock {\em Colloq. Math.}, 60/61(2):601--628, 1990.

\bibitem[CL06]{coifman2006diffusion}
Ronald~R Coifman and St{\'e}phane Lafon.
\newblock Diffusion maps.
\newblock {\em Applied and computational harmonic analysis}, 21(1):5--30, 2006.

\bibitem[Cro18]{Croydon2018}
D.~A. Croydon.
\newblock Scaling limits of stochastic processes associated with resistance
  forms.
\newblock {\em Ann. Inst. Henri Poincar\'{e} Probab. Stat.}, 54(4):1939--1968,
  2018.

\bibitem[DW05]{donglei2005laplacian}
Tang Donglei and SU~Weiyi.
\newblock The laplacian on the level 3 sierpinski gasket via the method of
  averages.
\newblock {\em Chaos, Solitons \& Fractals}, 23(4):1201--1209, 2005.

\bibitem[DW07]{donglei2007laplacian}
Tang Donglei and Su~Weiyi.
\newblock The {L}aplacian on $\beta$-sets via the method of averages.
\newblock {\em Chaos, Solitons \& Fractals}, 31(1):147--154, 2007.

\bibitem[FOT11]{FOT}
Masatoshi Fukushima, Yoichi Oshima, and Masayoshi Takeda.
\newblock {\em Dirichlet forms and symmetric {M}arkov processes}, volume~19.
\newblock Walter de Gruyter, 2011.

\bibitem[GBT21]{green2021minimax}
Alden Green, Sivaraman Balakrishnan, and Ryan~J Tibshirani.
\newblock Minimax optimal regression over {S}obolev spaces via {L}aplacian
  eigenmaps on neighborhood graphs.
\newblock {\em Preprint arXiv:2111.07394}, 2021.

\bibitem[GK06]{gine2006empirical}
Evarist Gin{\'e} and Vladimir Koltchinskii.
\newblock Empirical graph {L}aplacian approximation of {L}aplace-{B}eltrami
  operators: large sample results.
\newblock {\em Lecture Notes-Monograph Series}, pages 238--259, 2006.

\bibitem[GTGHS20]{garcia2020error}
Nicol{\'a}s Garc{\'\i}a~Trillos, Moritz Gerlach, Matthias Hein, and Dejan
  Slep{\v{c}}ev.
\newblock Error estimates for spectral convergence of the graph laplacian on
  random geometric graphs toward the {L}aplace--{B}eltrami operator.
\newblock {\em Foundations of Computational Mathematics}, 20(4):827--887, 2020.

\bibitem[HK98]{Kumagai1998}
B.~M. Hambly and T.~Kumagai.
\newblock Heat kernel estimates and homogenization for asymptotically
  lower-dimensional processes on some nested fractals.
\newblock {\em Potential Anal.}, 8(4):359--397, 1998.

\bibitem[Hut81]{Hutchinson81}
John~E. Hutchinson.
\newblock Fractals and self-similarity.
\newblock {\em Indiana Univ. Math. J.}, 30(5):713--747, 1981.

\bibitem[JMS08]{jones2008manifold}
Peter~W Jones, Mauro Maggioni, and Raanan Schul.
\newblock Manifold parametrizations by eigenfunctions of the {L}aplacian and
  heat kernels.
\newblock {\em Proceedings of the National Academy of Sciences},
  105(6):1803--1808, 2008.

\bibitem[KG00]{koltchinskii2000random}
Vladimir Koltchinskii and Evarist Gin{\'e}.
\newblock Random matrix approximation of spectra of integral operators.
\newblock {\em Bernoulli}, pages 113--167, 2000.

\bibitem[Kig01]{Kigamibook}
Jun Kigami.
\newblock {\em Analysis on fractals}.
\newblock Number 143. Cambridge University Press, 2001.

\bibitem[Kig12]{Kigami2012}
Jun Kigami.
\newblock Resistance forms, quasisymmetric maps and heat kernel estimates.
\newblock {\em Mem. Amer. Math. Soc.}, 216(1015):vi+132, 2012.

\bibitem[KK96]{Kumagai1996}
T.~Kumagai and S.~Kusuoka.
\newblock Homogenization on nested fractals.
\newblock {\em Probab. Theory Related Fields}, 104(3):375--398, 1996.

\bibitem[KL16]{KoltchinskiiLounici}
Vladimir Koltchinskii and Karim Lounici.
\newblock Asymptotics and concentration bounds for bilinear forms of spectral
  projectors of sample covariance.
\newblock {\em Ann. Inst. Henri Poincar\'{e} Probab. Stat.}, 52(4):1976--2013,
  2016.

\bibitem[Kol98]{Koltchinskii}
Vladimir Koltchinskii.
\newblock Asymptotics of spectral projections of some random matrices
  approximating integral operators.
\newblock In {\em High dimensional probability (Oberwolfach, 1996)}, volume~43
  of {\em Progr. Prob.}, pages 191--227. Birkhäuser Verlag, Basel, 1998.

\bibitem[Kum04]{Kumagai2004}
Takashi Kumagai.
\newblock Homogenization on finitely ramified fractals.
\newblock In {\em Stochastic analysis and related topics in {K}yoto}, volume~41
  of {\em Adv. Stud. Pure Math.}, pages 189--207. Math. Soc. Japan, Tokyo,
  2004.

\bibitem[Lin90]{Lindstrom}
Tom Lindstr{\o}m.
\newblock {\em Brownian motion on nested fractals}.
\newblock Number 420. American Mathematical Soc., 1990.

\bibitem[LS22]{lin2022varadhan}
Chen-Yun Lin and Christina Sormani.
\newblock From {V}aradhan's limit to eigenmaps: A guide to the geometric
  analysis behind manifold learning.
\newblock {\em Preprint arXiv:2210.10405}, 2022.

\bibitem[RBDV10]{rosasco-belkin-2010learning}
Lorenzo Rosasco, Mikhail Belkin, and Ernesto De~Vito.
\newblock On learning with integral operators.
\newblock {\em Journal of Machine Learning Research}, 11(2), 2010.

\bibitem[RS81]{ReedSimon}
Michael Reed and Barry Simon.
\newblock {\em I: Functional analysis}, volume~1.
\newblock Academic press, 1981.

\bibitem[Str01]{strichartz2001laplacian}
Robert~S Strichartz.
\newblock The {L}aplacian on the {S}ierpinski gasket via the method of
  averages.
\newblock {\em Pacific Journal of Mathematics}, 201(1):241--257, 2001.

\bibitem[VS23]{Venkatraman2023}
Raghavendra Venkatraman and Dejan Slep{\v{c}}ev.
\newblock Discrete to continuum limits in graph-{L}aplacian-based fractional
  {S}obolev spaces.
\newblock {\em In preparation}, 2023.

\bibitem[Wah24]{wahl2024kernel}
Martin Wahl.
\newblock A kernel-based analysis of {L}aplacian eigenmaps.
\newblock {\em arXiv preprint arXiv:2402.16481}, 2024.

\end{thebibliography}

\end{document}